\documentclass{article}

\PassOptionsToPackage{numbers}{natbib}
\usepackage[preprint]{neurips_2020}

\usepackage[utf8]{inputenc} 
\usepackage[T1]{fontenc}    
\usepackage{hyperref}       
\usepackage{url}            
\usepackage{booktabs}       
\usepackage{amsfonts}       
\usepackage{nicefrac}       
\usepackage{microtype}      
\usepackage{amsmath,amssymb}
\usepackage{amsthm}
\usepackage{times}
\usepackage{multicol} 
\usepackage{bm}
\usepackage{framed}
\usepackage{graphicx} 
\usepackage{subfigure} 
\usepackage{subfigmat}

\usepackage{xr}

\usepackage[numbers]{natbib}

\usepackage{algorithmic}
\usepackage[boxruled]{algorithm2e}

\usepackage{color}
\usepackage{array}
\usepackage{bbm}

\usepackage[symbol]{footmisc}

\theoremstyle{theorem}
\newtheorem{thm}{Theorem}[section]
\newtheorem{lem}[thm]{Lemma}
\newtheorem{prop}[thm]{Proposition}

\theoremstyle{definition}

\newtheorem{assump}{Assumption}
\theoremstyle{remark}
\newtheorem*{rem}{Remark}

\newtheorem*{pr_s_sgd_ef_convex_bound_prop}{Proof of Proposition \ref{prop: s_sgd_ef_convex_bound}}
\newtheorem*{pr_s_sgd_ef_cum_error_bound_prop}{Proof of Proposition \ref{prop: s_sgd_ef_cum_error_bound}}
\newtheorem*{pr_s_sgd_ef_strongly_convex_thm}{Proof of Theorem \ref{thm: s_sgd_ef_strongly_convex}}

\newtheorem*{pr_s_sgd_ef_general_nonconvex_thm}{Proof of Theorem \ref{thm: s_sgd_ef_general_nonconvex}}

\newtheorem*{pr_s_sgd_ef_final_obj_gap_bound_prop}{Proof of Proposition \ref{prop: s_snag_ef_final_obj_gap_bound}}
\newtheorem*{pr_s_snag_ef_m_t_bound_prop}{Proof of Proposition \ref{prop: s_snag_ef_m_t_bound}}
\newtheorem*{pr_s_snag_ef_strongly_convex_thm}{Proof of Theorem \ref{thm: s_snag_ef_strongly_convex}}
\newtheorem*{pr_reg_s_snag_ef_general_nonconvex_thm}{Proof of Theorem \ref{thm: reg_s_snag_ef_general_nonconvex}}

\title{Accelerated Sparsified SGD with Error Feedback}

\author{
  Tomoya Murata\\
  NTT DATA Mathematical Systems Inc. , Tokyo, Japan\\
  Department of Mathematical Informatics, \\
  Graduate School of Information Science and Technology, \\
  The University of Tokyo, Tokyo, Japan \\
  \texttt{murata@msi.co.jp} \\
  \And
  Taiji Suzuki\\
  Department of Mathematical Informatics, \\
  Graduate School of Information Science and Technology, \\
  The University of Tokyo, Tokyo, Japan \\
  Center for Advanced Intelligence Project, RIKEN, Tokyo, Japan \\
  \texttt{taiji@mist.i.u-tokyo.ac.jp} 
}

\begin{document}

\maketitle

\begin{abstract}
A stochastic gradient method for synchronous distributed optimization is studied. For reducing communication cost, we particularly focus on utilization of compression of communicated gradients. Several work has shown that {\it{sparsified}} stochastic gradient descent method (SGD) with {\it{error feedback}} asymptotically achieves the same rate as (non-sparsified) parallel SGD. However, from a viewpoint of non-asymptotic behavior, the compression error may cause slower convergence than non-sparsified SGD in early iterations. This is problematic in practical situations since early stopping is often adopted to maximize the generalization ability of learned models. 
For improving the previous results, we propose and theoretically analyse a sparsified stochastic gradient method with error feedback scheme combined with {\it{Nesterov's acceleration}}. 
It is shown that the necessary per iteration communication cost for maintaining the same rate as vanilla SGD can be smaller than non-accelerated methods in convex and even in nonconvex optimization problems. This indicates that our proposed method makes a better use of compressed information than previous methods. Numerical experiments are provided and empirically validates our theoretical findings. 
\end{abstract}

\section{Introduction}
In typical modern machine learning tasks, we often encounter large scale optimization problems, which require huge computational time to solve. Hence, saving computational time of optimization processes is practically quite important and is main interest in the optimization community.  \par
To tackle large scale problems, a golden-standard approach is the usage of {\it{Stochastic Gradient Descent}} (SGD) method \cite{robbins1951stochastic}. For reducing loss, SGD updates the current solution by using a stochastic gradient in each iteration, that is the average of the gradients of the loss functions correspond to a random subset of the dataset (mini-batch) rather than the whole dataset. This (stochastic) mini-batch approach allows that SGD can be faster than deterministic full-batch methods in terms of computational time \cite{dekel2012optimal, li2014efficient}. Furthermore, {\it{Stochastic Nesterov's Accelerated Gradient}} (SNAG) method and its variants have been proposed \cite{hu2009accelerated, chen2012optimal, ghadimi2016accelerated}, that are based on the combination of SGD with Nesterov's acceleration \cite{nesterov2013introductory, nesterov2013gradient, tseng2008accelerated}. Mini-batch SNAG theoretically outperforms vanilla mini-batch SGD for moderate optimization accuracy, though its asymptotic convergence rate matches that of SGD. \par
For realizing further scalability, {\it{distributed optimization}} have received much research attention \cite{bekkerman2011scaling, duchi2011dual, jaggi2014communication, gemulla2011large, dean2012large, ho2013more, arjevani2015communication, chen2016revisiting, goyal2017accurate}. Distributed optimization methods are mainly classified as synchronous centralized \cite{zinkevich2010parallelized, dekel2012optimal, shamir2014distributed},  asynchronous centralized \cite{recht2011hogwild, agarwal2011distributed, lian2015asynchronous, liu2015asynchronous, zheng2017asynchronous}, synchronous decentralized  \cite{nedic2009distributed, yuan2016convergence, lian2017can, lan2017communication, uribe2017optimal, scaman2018optimal} and asynchronous decentralized \cite{lian2017asynchronous, lan2018asynchronous} ones by their communication types. In this paper, we particularly focus on data parallel stochastic gradient methods for {\it{synchronous centralized}} distributed optimization with smooth objective function $F: \mathbb{R}^d \to \mathbb{R}, F(x) = \frac{1}{P}\sum_{p=1}^P\frac{1}{N}\sum_{i=1}^N f_{i, p}(x)$, where each $\{f_{i, p}\}_{i=1}^N$ corresponds to a data partition of the whole dataset for the $p$-th node (or processor). In this setting, first each processor $p$ computes a stochastic gradient of $(1/N)\sum_{i=1}^N f_{i, p}(x)$ and then the nodes send the gradients each other. Finally, the current solution is updated using the averaged gradient on each processor. Here we assume that node-to-node broadcasts are used, but it is also possible to utilize an intermediate parameter server. \par
A main concern in synchronous distributed optimization is communication cost because it can easily be a bottleneck in optimization processes. Theoretically, naive parallel mini-batch SGD achieves linear speed up with respect to the number of processors \cite{dekel2012optimal, li2014efficient}, but not empirically due to this cost \cite{shamir2014distributed, chen2016revisiting}. For leveraging the power of parallel computing, it is essential to reduce the communication cost. \par
One of fascinating techniques for reducing communication cost in distributed optimization is compression of the communicated gradients \cite{aji2017sparse, lin2017deep, wangni2018gradient, alistarh2018convergence, stich2018sparsified, shi2019distributed, karimireddy2019error, seide20141, wen2017terngrad, alistarh2017qsgd, wu2018error}. {\it{Sparsification}} is an approach in which the gradient is compressed by sparsifying it in each local node before communication \cite{aji2017sparse, lin2017deep, wangni2018gradient, alistarh2018convergence, stich2018sparsified, shi2019distributed, karimireddy2019error}. For sparsifying a gradient, top-$k$ algorithm, that drops the $d-k$ smallest components of the gradient by absolute value from the $d$ components of the gradient, has been typically used. 
Another example of compression is quantization, which is a technique that limit the number of bits to represent the communicated gradients. Several work has demonstrated that parallel SGD with quantized gradients has good practical performance \cite{seide20141, wen2017terngrad, alistarh2017qsgd, wu2018error}. 
Particularly, Alistarh et al. \cite{alistarh2017qsgd}  have proposed Quantized SGD (QSGD), which is the first quantization algorithm with a theoretical convergence rate. QSGD is based on unbiased quantization of the communicated gradient. \par

However, theoretically there exists an essential trade-off between communication cost and convergence speed when we use naive gradient compression schemes. Specifically, naive compression (including sparsification and quantization) causes large variances and theoretically always slower than vanilla SGD, though they surely reduce the communication cost \cite{stich2018sparsified, alistarh2017qsgd}. \par
{\it{Error feedback}} scheme partially solves this trade-off problem. Some work has considered the usage of compressed gradients with the locally accumulated compression errors in each node and its effectiveness has been validated empirically \cite{aji2017sparse, lin2017deep, wu2018error}. Very recently, several work has attempted to analyse and justified the effectiveness of error feedback in a theoretical view \cite{alistarh2018convergence, stich2018sparsified, cordonnier2018convex, karimireddy2019error, tang2019doublesqueeze, zheng2019communication}. Surprisingly, it has been shown that Sparsified SGD with error feedback {\it{asymptotically}} (in terms of optimization accuracy) achieves the {\it{same}} rate as non-sparsified SGD. \par

Nevertheless, for a theoretical point of view, the method still may require large communication cost for maintaining the rate of ideal SGD particularly in early iterations due to the compression error. Therefore, more communication efficient method are desired. The goal of this paper is a creation of an algorithm that requires smaller per iteration communication cost than sparsified SGD with error feedback while it maintains the same rate as vanilla SGD.


\paragraph*{Main contribution} We construct and analyse Sparsified Stochastic Nesterov's Accelerated Gradient method (S-SNAG-EF) based on the combination of (i) unbiased sparsification of the stochastic gradients; (ii) error feedback scheme; and (iii) Nesterov's acceleration technique. The main message of this paper is the following: 
\begin{framed}
S-SNAG-EF maintains the convergence rate of vanilla SGD with per iteration communication cost $O(P^{3/4}\varepsilon d)$\footnotemark for general convex problems. In contrast, non-accelerated methods require $O(P^{1/2}\sqrt{\varepsilon}d)$. Namely, the per iteration communication cost of our method has a better dependence on $\varepsilon$ than previous methods. Also, this superiority is true even in nonconvex problems. 
\end{framed}
\footnotetext{Here, $P$ is the number of processors, $\varepsilon$ is desired optimization accuracy and $d$ is the dimension of the parameter space. $\varepsilon \leq 1/P$ is assumed for simplicity. }
We also give thorough analysis of non-accelerated sparsified SGD with error feedback based on {\it{unbiased}} random compression (we call this algorithm as S-SGD-EF in this paper) and show better results than previously known ones\footnotemark. 
A comparison of our method with the most relevant previous methods is summarized in Table \ref{table: commu_cost_comparison}. \par
 \par
\footnotetext{These improvements come from the unbiasedness of the random compression. Several previous work have given analysis of (non-accelerated) sparsified SGD with errorfeedback in distributed settings based on more general compression scheme including the unbiased random compression\citep{alistarh2018convergence,cordonnier2018convex,tang2019doublesqueeze,zheng2019communication}. However, these approaches do not fully utilize the unbiasedness of compression and the convergence rates in parallel settings are worse than ours. }

\begin{table}[]
    \centering
    \scalebox{0.94}{
    \label{table: commu_cost_comparison}
    \begin{tabular}{c c c c}
        \hline
         & \hspace{-1em}general convex & \hspace{-1em}strongly convex & \hspace{-1em}general nonconvex \\ \hline 
         S-SGD & $d$ & $d$ & $d$  \\
         S-SNAG & $d$ & $d$ & $d$ \\ 
         \begin{tabular}{c} MEM-SGD \cite{ cordonnier2018convex}\end{tabular}& $(P\sqrt{\varepsilon}\wedge 1) d$ & $(P\sqrt{\varepsilon}\wedge 1) d$ & No Analysis \\
         \begin{tabular}{c} DoubleSqueeze 
         \cite{tang2019doublesqueeze} \end{tabular} & No Analysis & No Analysis & $(P\sqrt{\varepsilon}\wedge 1) d$ \\
         {\color{red}S-SGD-EF} & {\color{red}$(\sqrt{P\varepsilon}\wedge 1)d$ }& {\color{red}$(\sqrt{P\varepsilon}\wedge 1)d$} & {\color{red}$(\sqrt{P\varepsilon}\wedge 1)d$} \\
         {\color{red}S-SNAG-EF} & {\color{red}$(P^{\frac{3}{4}}\varepsilon\wedge 1)d$} & {\color{red}$((P^{\frac{1}{3}}\mu^{\frac{1}{3}}\varepsilon^{\frac{2}{3}} + P^\frac{3}{4}\mu^\frac{1}{4}\varepsilon^\frac{3}{4})\wedge 1)d$} & {\color{red}$((P^\frac{3}{4}\varepsilon^\frac{3}{4} + P^\frac{4}{3}\varepsilon^\frac{2}{3})\wedge 1)d$}\\ \hline
    \end{tabular}
    }
        \caption{Comparison of the order of the per iteration communication cost for maintaining the rate of vanilla SGD. Here, we refer to the standard SGD iteration complexities for achieving $F(x) - F(x_*) \leq \varepsilon$ (for convex objectives) or $\|\nabla F(x)\|^2 \leq \varepsilon$ (for nonconvex ones), that are $O(1/\varepsilon + 1/(P\varepsilon^2))$ for general convex problems, $O(1/\mu + 1/(P\mu\varepsilon))$ for $\mu$-strongly convex ones and $O(1/\varepsilon + 1/(P\varepsilon^2))$ for general nonconvex ones. The per iteration communication cost is simply computed by $\widetilde{O}(kP \wedge d)$, where $d$ is the dimension of the parameter space, $k/d$ is the compression ratio for sparsified methods and $P$ is the number of processors. For simple comparison, we  assume that $\varepsilon \leq O(1/P)$. Also $L$, $\mathcal{V}$, $\Delta = F(x_{\mathrm{ini}}) - F(x_*)$, $D = \|x_{\mathrm{ini}} - x_*\|^2$ are assumed to be $\Theta(1)$. Extra logarithmic factors are ignored.
    }
\end{table}
\setlength{\textfloatsep}{8pt}
{\bf{Related Work}}\ \ \ 
We briefly describe the most relevant papers to this work. Stich et al. \cite{stich2018sparsified} have first provided theoretical analysis of sparsified SGD with error feedback (called MEM-SGD) and shown that MEM-SGD asymptotically achieves the rate of non-sparsifed SGD. However, their analysis is limited in serial computing settings, i.e., $P = 1$. Independently, Alistarh et al. \cite{alistarh2018convergence} have also theoretically considered sparsified SGD with error feedback in parallel settings for convex and nonconvex objectives. However, their analysis is still unsatisfactory because their analysis relies on an artificial assumption due to the usage of top-$k$ algorithm for gradient compression and it is unclear from their results whether the algorithm asymptotically possesses the linear speed up property with respect to the number of nodes. After their work, Cordonnier et al. \cite{cordonnier2018convex} have analyze sparsified SGD with error feedback in parallel settings and shown the linear speedup property at the stochastic error term, but not the compression error terms. Recently, Karimireddy et al. \cite{karimireddy2019error} have also analysed a variant of sparsified SGD with error feedback (called EF-SGD) for convex and nonconvex cases in serial computing settings. Differently from ours, their analysis allows non-smoothness of the objectives for convex cases, though the convergence rate is always worse than vanilla SGD and the algorithm does not possesses the asymptotic optimality.
More recently, Tang et al. \citep{tang2019doublesqueeze} have proposed and analysed Doublesqueeze in parallel and nonconvex settings. In Doublesqueeze, error feedback scheme is applied to each worker and also the parameter sever. They have also shown the linear speedup property at the stochastic error term, but not the compression error terms. Zheng et al. \citep{zheng2019communication} have proposed blockwise compression with error feedback and its acceleration by Nesterov's momentum. They have analysed the algorithms in parallel and nonconvex settings but the convergence rates are essentially same as Doublesqueeze. Importantly, they have not shown any theoretical superiority of their accelerated method to non-accelerated one.  

\section{Notation and Assumptions}
$\| \cdot \|$ denotes the Euclidean $L_2$ norm $\| \cdot \|_2$: $\|x\| = \|x\|_2 = \sqrt{\sum_{i}x_i^2}$. For natural number $m$, $[m]$ denotes the set $\{1, 2, \ldots, m\}$. 
We define $Q(z): \mathbb{R}^d \to \mathbb{R}$ as the quadratic function with center $z$, i.e., $Q(z)(x) = \|x - z\|^2$. A sparsification operator $\mathrm{RandComp}$ is defined as $\mathrm{RandComp}(x, k)j = (d/k)x_j$ for $j$ in a uniformly random subset $J$ with $\#J=k$ and $\mathrm{RandComp}(x, k)j = 0$ otherwise. 

The followings are theoretical assumptions for our analysis. These are very standard in optimization literature. We always assume the first three assumptions.
\begin{assump}\label{assump: sol_existence}
$F$ has a minimizer $x_* \in \mathbb{R}^d$. 
\end{assump}
\begin{assump}\label{assump: smoothness}
$F$ is $L$-smooth ($L > 0$), i.e., $\|\nabla F(x) - \nabla F(y)\| \leq L\|x-y\|, \forall x, y \in \mathbb{R}^d$.
\end{assump}
\begin{assump}\label{assump: bounded_variance}
$\{f_{i, p}\}_{i, p}$ has $\mathcal{V}$-bounded variance, i.e., $\frac{1}{NP}\sum_{i, p}\|\nabla f_{i, p}(x) - F(x)\|^2 \leq \mathcal{V}, \forall x \in \mathbb{R}^d$.
\end{assump}
\begin{assump}\label{assump: strong_convexity}
$F$ is $\mu$-strongly convex ($\mu >0$), i.e., $F(y) - (F(x) + \langle \nabla F(x), y - x\rangle) \geq (\mu/2)\|x - y\|^2, \forall x, y \in \mathbb{R}^d$.
\end{assump}

\section{Proposed Algorithm}
In this section, we first illustrate three core technique for constructing our algorithm. Then we describe our proposed algorithm S-SNAG-EF. \par
{\bf{Sparsifcation}}\ \ \ Gradient sparsification is quite intuitive approach for reducing communication cost in distributed optimization. Concretely, $k$ elements of local stochastic gradient over $d$ ones are selected and set the other ones to be zero on each processor. Then, the sparsified gradients are communicated between the processors. Several selection methods have been proposed, but we adopt random sparsification, that is the simplest one and desirable from a theoretical point of view because of its unbiasedness. It is known that this naive sparsification causes $d/k$-times larger variances and hence $d/k$-times slower convergence than vanilla SGD. \par
{\bf{Errorfeedback}}\ \ \ The key observation is that each processor can make use of the history of its local, but not compressed stochastic gradients for correcting the compression. For the first update, each processor sparsifies its local gradient and communicate it. Then, each processor broadcasts it to the other processors and aggregates the received gradients. The difference from naive sparsification is that each processor saves the difference of the non-compressed gradient from the compressed gradient (we call this as compression error) cumulatively. During subsequent updates, each processor sparsifies the sum of its local gradient and appropriately scaled cumulative compression error rather than the former only. We call this process as error feedback. The formal description of the algorithm are given in Algorithm \ref{alg: s_sgd_ef}. It is known that sparsified SGD with error feedback asymptotically achieves the same rate as vanilla SGD. 

\begin{algorithm}[H]
\label{alg: s_sgd_ef}
\caption{S-SGD-EF($F$, $x_{\mathrm{in}}$, $\{\eta_t\}_{t=1}^{\infty}$, $\gamma$, $k$, $T$)}
\begin{algorithmic}[1]
\STATE Set: $x_0 = x_{\mathrm{in}}$, $m_{0, p} = 0\ (p \in [P])$.
\FOR {$t=1$ to $T$}
\FOR {$p=1$ to $P$ \it{in parallel}}
\STATE Compute i.i.d. stochastic gradient of the partition of $F$: $\nabla f_{i, p}(x_{t-1})$.
\STATE Correct gradient based on cumulative compression error: $g_{t, p} = \nabla f_{i, p}(x_{t-1}) + (\gamma/\eta_t) m_{t-1, p}$.
\STATE Compress: $\bar g_{t, p} = \mathrm{RandComp}(g_{t, p}, k)$.
\STATE Update cumulative compression error: \\
$m_{t, p} = m_{t-1, p} + \eta_t(\nabla f_{i, p}(x_{t-1}) - \bar g_{t, p})$.
\ENDFOR
\STATE Broadcast and Receive: $\bar g_{t, p}\ (p \in [P])$.
\FOR {$p=1$ to $P$ \it{in parallel}}
\STATE Update solution: $x_t = x_{t-1} - \eta_t \frac{1}{P}\sum_{p=1}^P \bar g_{t, p}$.
\ENDFOR
\ENDFOR
\ENSURE $x_{\hat t}$. 
\end{algorithmic}
\end{algorithm}

\begin{rem}[Difference from previous algorithms]
Algorithm \ref{alg: s_sgd_ef} can be regard as an extension of Mem-SGD \cite{stich2018sparsified} or EF-SGD \cite{karimireddy2019error} to parallel computing settings, though these two methods mainly utilize top-$k$ compression for gradient sparsification. In constrast, we rather use unbiased random compression. This difference is essential for our analysis.
\end{rem}
{\bf{Acceleration}}\ \ \ It is well-known that Nesterov's accelerated method achieves faster convergence than vanilla GD and is optimal in convex optimization. Hence it is also expected that the acceleration is effective to improve compressed stochastic gradient methods. The most famous form of the acceleration algorithm  uses momentum: the solution is constructed as the sum of the standard gradient descent solution and the appropriately scaled momentum, that is the difference of current solution from the previous one. In an alternative form of the acceleration algorithm, we use three updated solutions, that are (i) a conservative solution (because step size is small, but updated from (iii) at each iteration); (ii) an aggressive solution (i.e., because of large step size); (iii) the convex combination of (i) and (ii). At first sight, they seem to be different algorithms, but it is easy to show the equivalency. We will adopt the latter for our algorithm. The concrete procedure of NAG is illustrated in Algorithm \ref{alg: nag}. 

\begin{algorithm}[t]
\label{alg: one_iter_nag}
\caption{OneIterNAG($x$, $y$, $z$, $\Delta_y$, $\Delta_z$, $\alpha$, $\beta$)}
\begin{algorithmic}[1]
\STATE $y = y - \Delta_y$.
\STATE $z = (1-\beta)z + \beta x - \Delta_z$.
\STATE $x = (1 - \alpha)y + \alpha z$.
\ENSURE $x, y, z$
\end{algorithmic}
\end{algorithm}
\begin{algorithm}[t]
\label{alg: nag}
\caption{NAG($F$, $x_{\mathrm{in}}$, $\{\eta_t, \lambda_t, \alpha_t, \beta_t\}_{t=1}^{\infty}$, $T$)}
\begin{algorithmic}[1]
\STATE Set: $y_0 = z_0 = x_{\mathrm{in}}$.
\FOR {$t=1$ to $T$}
\STATE $x_t, y_t, z_t = $ OneIterNAG($x_{t-1}$, $y_{t-1}$, $z_{t-1}$, $\eta_t \nabla F(x_{t-1})$, $\lambda_t \nabla F(x_{t-1})$, $\alpha_t$, $\beta_t$).
\ENDFOR
\ENSURE $y_T$
\end{algorithmic}
\end{algorithm}

{\bf{Proposed Algorithm: S-SNAG-EF}}\ \ \ The procedure of S-SNAG-EF for convex objectives is provided in Algorithm \ref{alg: s_snag_ef}. At first, each processor computes an i.i.d. stochastic gradient in line 4. In line 5 and 6, two different gradient estimators by randomly picking $k/2$-coordinates for each are computed. Also in line 7, we update three cumulative compression errors. Why are different compressed estimators and cumulative errors necessary for appropriate updates? In a typical acceleration algorithm we construct two different solution paths $\{y_t\}$ and $\{z_t\}$, and their aggregations $\{x_t\}$ as in line 11. The aggregation of the "conservative" solution $y_t$ (because of small learning rate $\eta_t$) and "aggressive" solution $z_t$ (because of large learning rate $\lambda_t$) is the essence of Nesterov's acceleration. On the other hand, from a theoretical point of view, the impact of error feedback to the vanilla stochastic gradient should be scaled to the inverse of learning rate as in line 5. Therefore, for using two different learning rates, it is necessary to construct two compressed gradient estimators and hence three compression errors. Finally, we update three solutions similar to the Nesterov's accelerated algorithm. \par

\begin{algorithm}[t]
\label{alg: s_snag_ef}
\caption{S-SNAG-EF($F$, $x_{\mathrm{in}}$,  $\{\eta_t, \lambda_t, \alpha_t, \beta_t\}_{t=1}^{\infty}$, $\gamma$, $k$, $T$)}
\begin{algorithmic}[1]
\STATE Set: $y_0 = z_0 = x_{\mathrm{in}}$, $m_{0, p} = m_{0, p}^{(y)} = m_{0, p}^{(z)} = 0\ (p \in [P])$.
\FOR {$t=1$ to $T$}
\FOR {$p=1$ to $P$ \it{in parallel}}
\STATE Compute i.i.d. stochastic gradient of the partition of $F$: 
$\nabla f_{i, p}(x_{t-1})$.
\STATE Correct gradients based on cumulative compression errors: $g_{t, p}^{(y)} = g_{t, p}^{(z)} = \nabla f_{i, p}(x_{t-1})$, \\
   $\begin{cases}g_{t, p}^{(y)} += (\gamma/\eta_t) m_{t-1, p}, \\
    g_{t, p}^{(z)} += (\gamma/\lambda_t)((1-\beta_t) m_{t-1, p}^{(z)}+\beta_t m_{t-1, p}). \end{cases}$
\STATE Compress corrected gradients: \\
    $\begin{cases}\bar g_{t, p}^{(y)} = \mathrm{RandComp}(g_{t, p}^{(y)}, k/2), \\
    \bar g_{t, p}^{(z)} = \mathrm{RandComp}(g_{t, p}^{(z)}, k/2). \end{cases}$
\STATE Update cumulative compression errors: \\
$m_{t, p}, m_{t, p}^{(y)}, m_{t, p}^{(z)} =$ OneIterNAG($m_{t-1, p}$, $m_{t-1, p}^{(y)}$, $m_{t-1, p}^{(z)}$, $\Delta_y^t$, $\Delta_z^t$, $\alpha_t$, $\beta_t$), where $\Delta_y^t = \eta_t(\bar g_{t, p}^{(y)} - \nabla f_{i, p}(x_{t-1}))$ and $\Delta_z^t = \lambda_t(\bar g_{t, p}^{(z)} - \nabla f_{i, p}(x_{t-1}))$.
\ENDFOR
\STATE Broadcast and Receive: \\
$\bar g_{t, p}^{(y)}, \bar g_{t, p}^{(z)}\ (p \in [P])$.
\FOR {$p=1$ to $P$ \it{in parallel}}
\STATE Update solutions: \\
$x_t, y_t, z_t =$ OneIterNAG($x_{t-1}$, $y_{t-1}$, $z_{t-1}$, $\eta_t \frac{1}{P}\sum_{p=1}^P \bar g_{t, p}^{(y)}$, $\lambda_t \frac{1}{P}\sum_{p=1}^P \bar g_{t, p}^{(z)}$, $\alpha_t$, $\beta_t$)
\ENDFOR
\ENDFOR
\RETURN $x_{\mathrm{out}} = y_{\hat t}$. 
\end{algorithmic}
\end{algorithm}

\begin{rem}[Parameter tuning]
It seems that Algorithm \ref{alg: s_snag_ef} has many tuning parameters. However, this is not. Specifically, as Theorem \ref{thm: s_snag_ef_strongly_convex} in Section \ref{sec: convergence_analysis} indicates, actual tuning parameters are only constant learning rate $\eta$, strong convexity $\mu$ and $\gamma$, and the other parameters are theoretically determined. This means that the additional tuning parameters compared to S-SGD-EF are essentially only strong convexity parameter $\mu$. Practically, fixing $\gamma = 0.5 \times k/d$ works well.
\end{rem}

\section{Convergence Analysis}\label{sec: convergence_analysis}
In this section, we provide convergence analysis of our proposed S-SNAG-EF. For the space limitation, we give the analysis of S-SGD-EF in Section \ref{sec: analysis_s_sgd_ef} of the supplementary material. For convex cases, we always assume the strong convexity of the objective in this paper\footnotemark. \footnotetext{For non-strongly convex cases, we can immediately derive the convergence rate from the ones for strongly convex cases by taking standard dummy regularizer approach and we omit it here.} \par 
Let $m_t$ be the mean of the cumulative compression errors of the all nodes at $t$-th iteration, i.e., $m_t = (1/P)\sum_{p=1}^P m_{t, p}$. We use $\widetilde{O}$ notation to hide additional logarithmic factors for simplicity. For the proofs of the statements, see  Section \ref{sec: analysis_s_snag_ef} of the supplementary material. \par

The following proposition holds for strongly convex objective $F$.
\begin{prop}[Strongly convex]\label{prop: s_snag_ef_final_obj_gap_bound}
Suppose that Assumptions \ref{assump: sol_existence}, \ref{assump: smoothness}, \ref{assump: bounded_variance} and \ref{assump: strong_convexity} hold. Let $\eta_t = \eta \leq 1/(2L)$, $\lambda_t = \lambda = (1/2)\sqrt{\eta/\mu}$,  $\alpha_t = \alpha = \lambda\mu/(2 + \lambda\mu)$ and $\beta_t = \beta = \lambda\mu/(1 + \lambda\mu)$. Then S-SNAG-EF satisfies
\begin{align*}
    \mathbb{E}[F(x_{\mathrm{out}}) - F(x_*)] 
    \leq&\ \Theta\left(\mu(1-\sqrt{\eta\mu})^T\|x_0 - x_*\|^2 + \sqrt{\frac{\eta}{\mu}}\frac{\mathcal{V}^2}{P} \right). \\
    &\left. + \sum_{t=1}^T(1-\sqrt{\eta\mu})^{T-t}\left(\lambda L^2\mathbb{E}\|m_{t-1}\|^2 - \eta \mathbb{E}\|\nabla F(x_{t-1})\|^2\right) + L\mathbb{E}\|m_{T-1}\|^2\right),
\end{align*}
where $x_{\mathrm{out}} = x_{T-1}$.
\end{prop}
\begin{rem}
The first deterministic error term is scaled to $(1 - \sqrt{\eta\mu})^T$ rather than $(1 - \eta\mu)^T$ thanks to the acceleration scheme at the expense of $1/\sqrt{\eta\mu}$ times larger stochastic error (the second term) than the one of vanilla SGD. This evokes the bias-variance trade-off in the rate of  vanilla accelerated SGD. The third and last terms are the compression error caused by the gradient sparsification.
\end{rem}
The compression error terms are bounded by the following proposition.
\begin{prop}\label{prop: s_snag_ef_m_t_bound}
Suppose that Assumptions \ref{assump: bounded_variance} holds. Let  $\gamma = \Theta(k/d)$,  $\beta_t \leq \Theta(\gamma^3/\alpha_t^2)$ be sufficiently small and $\{\alpha_t\}$ is monotonically non-increasing. Then S-SNAG-EF satisfies
\begin{align*}
    \mathbb{E}\|m_t\|^2 \leq \Theta\left(\sum_{t'=1}^t\frac{(\eta_{t'}^2 + (\alpha_{t'}^2/\gamma^2)\lambda_{t'}^2)d}{kP}(1 - \gamma)^{t - t'}(\mathcal{V} + \mathbb{E}\|\nabla F(x_{t'-1})\|^2)\right).
\end{align*}
\end{prop}
\begin{rem}
This proposition shows that the cumulative compression error is bounded even if $t \to \infty$. This is the key property for obtaining the asymptotical rate of vanilla SGD. Also note hat the cumulative compression error has a factor $1/P$, this does not arise in any previous analysis. 
\end{rem}
Combining Proposition \ref{prop: s_snag_ef_final_obj_gap_bound} and \ref{prop: s_snag_ef_m_t_bound} yields the following theorem. 
\begin{thm}[Strongly convex]\label{thm: s_snag_ef_strongly_convex}
Suppose that Assumptions \ref{assump: sol_existence}, \ref{assump: smoothness}, \ref{assump: bounded_variance} and \ref{assump: strong_convexity} hold. 
Let $\lambda_t, \alpha_t$ and $\beta_t$ are the same ones in Proposition \ref{prop: s_snag_ef_final_obj_gap_bound} and $\gamma = \Theta(k/d)$ be sufficiently small. Then the iteration complexity $T$ of S-SNAG-EF with appropriate $\eta_t = \eta$  to acheive $\mathbb{E}[F(x_{\mathrm{out}}) - F(x_*)] \leq \varepsilon$ is 
\begin{align}
    \widetilde O&\left( \sqrt{\frac{L}{\mu}} + \frac{\mathcal{V}}{P}\frac{1}{\mu\varepsilon} + \frac{d}{k} + \frac{d^{\frac{3}{2}}}{k^{\frac{3}{2}}\sqrt{P}}\sqrt{\frac{L}{\mu}} + \frac{d^{\frac{4}{3}}}{k^{\frac{4}{3}}P^{\frac{1}{3}}}\frac{L^{\frac{2}{3}}}{\mu^{\frac{2}{3}}} + \frac{d}{kP^{\frac{1}{4}}}\frac{(L^2\mathcal{V})^{\frac{1}{4}}}{\mu^{\frac{3}{4}}\varepsilon^{\frac{1}{4}}}\right), \label{iter_comp: s_snag_ef}
\end{align}
where $x_{\mathrm{out}} = x_T$.
\end{thm}
\begin{rem}
In contrast, non-accelerated S-SGD-EF only achieves
\begin{align*}
     \widetilde{O}\left( \frac{L}{\mu} + \frac{\mathcal{V}}{P}\frac{1}{\mu\varepsilon}  +  \frac{d}{k} + \frac{d}{k\sqrt{P}}\left(\frac{L}{\mu} + \frac{\sqrt{L\mathcal{V}}}{\mu\sqrt{\varepsilon}}\right) \right).
\end{align*}
\end{rem}
{\bf{Asymptotic View: Same Rate as Non-compressed SGD}}\ \ \ As $\varepsilon \to 0$, the asymptotic rate of S-SNAG-EF becomes $1/(\mu\varepsilon)$, that is the convergence rate of the vanilla SGD. Several previous methods  and S-SGD-EF also possess this property. \par
{\bf{Non-asymptotic View: Less Compression Error than Non-accelerated Methods}} \ \ \ The compression error (third to last terms in (\ref{iter_comp: s_snag_ef}) has a better dependency on $\varepsilon$ than the one of S-SGD-EF. As a result, the necessary number of communicated components $k$ to maintain the rate of vanilla SGD $1/(\mu\varepsilon)$ is $(\widetilde O(P^{1/3}\mu^{1/3}\varepsilon^{3/4}+P^{3/4}\mu^{1/4}\varepsilon^{3/4}) \wedge 1)d$, which can be much better than the one of S-SGD-EF $(O(P^{1/2}\sqrt\varepsilon)\wedge 1 )d$ for moderate $\varepsilon > 0$. Here, for simple comparison, we assume that $\varepsilon \leq 1/P$.\par

\section{Extension to Nonconvex Cases}
In this section, we briefly discuss an extension of S-SNAG-EF (Algorithm \ref{alg: s_snag_ef}) to nonconvex cases. Unfortunately, Algorithm \ref{alg: s_snag_ef} has no theoretical guarantee for nonconvex cases generally. Hence We adopt the standard recursive regularization scheme to resolve this problem. Specifically, Algorithm \ref{alg: reg_s_snag_ef} repeatedly minimize the "regularized" objective $F + \sigma Q(x_{t-1}^{++})$ by using S-SNAG-EF, where $Q(x_{t-1}^{++})(x) = \|x - x_{t-1}^{++}\|^2$ and $x_{t-1}^{++}$ is the current solution. This means that the objective function is convexified by $L_2$-regularization around the current solution and the regularized objective is minimized by S-SNAG-EF at each iteration. We call this algorithm Reg-S-SNAG-EF (Algorithm \ref{alg: reg_s_snag_ef}). 

\begin{algorithm}[H]
\label{alg: reg_s_snag_ef}
\caption{Reg-S-SNAG-EF ($F$, $x_{\mathrm{in}}$, $\{\eta_t, \lambda_t, \alpha_t, \beta_t \}_{t=1}^{\infty}$, $\gamma$, $k$, $T$, $\sigma$, $S$)}
\begin{algorithmic}
\STATE Set: $x_0 = x_{\mathrm{in}}$.
\FOR {$s=1$ to $S$}
\STATE Run: $x_s = $ S-SNAG-EF($F + \sigma Q(x_{s-1})$, $x_{s-1}$, $\{\eta_t, \lambda_t, \alpha_t, \beta_t \}_{t=1}^{\infty}$, $\gamma$, $k$, $T$)
\ENDFOR
\ENSURE $x_{\hat s}$.
\end{algorithmic}
\end{algorithm}

\begin{thm}[General nonconvex]\label{thm: reg_s_snag_ef_general_nonconvex}
Suppose that Assumptions \ref{assump: sol_existence}, \ref{assump: smoothness} and \ref{assump: bounded_variance} hold. Let $\sigma = L$, and $\lambda_t$, $\alpha_t$, $\beta_t$ and $\gamma$ be the same ones in Theorem \ref{thm: s_snag_ef_strongly_convex} (with $\mu \leftarrow \sigma$), and $T = \widetilde\Theta(1/\sqrt{\eta L})$ and $S = \Theta(1+L\Delta/\varepsilon)$ be sufficiently large. Then the iteration complexity $ST$ of Reg-S-SNAG-EF with appropriate $\eta_t = \eta$ for achieving $\mathbb{E}\|\nabla F(x_{\mathrm{out}})\|^2 \leq \varepsilon$ is
\begin{align*}
    \widetilde O&\left( \frac{L\Delta}{\varepsilon} + \frac{\mathcal{V}}{P}\frac{L\Delta}{\varepsilon^2} + \left(\frac{d}{k} + \frac{d^{\frac{3}{2}}}{k^{\frac{3}{2}}\sqrt{P}} + \frac{d^{\frac{4}{3}}}{k^{\frac{4}{3}}P^{\frac{1}{3}}}\right)\frac{L\Delta}{\varepsilon} + \frac{d}{kP^{\frac{1}{4}}}\frac{L\mathcal{V}^{\frac{1}{4}}\Delta}{\varepsilon^{\frac{5}{4}}}\right),
\end{align*}
where $x_{\mathrm{out}} = x_{\hat s}$ and $\hat s \sim [S]$ according to $\{1/S\}_{s=1}^S$.
\end{thm}
\begin{rem}
In contrast, S-SGD-EF only achieves 
\begin{align*}
     O\left( \frac{L\Delta}{\varepsilon} + \frac{\mathcal{V}}{P}\frac{L\Delta}{\varepsilon^2} + \frac{d}{k} + \frac{d}{k\sqrt{P}} \left(\frac{L\Delta}{\varepsilon} + \frac{L\sqrt{\mathcal{V}}\Delta}{\varepsilon^{\frac{3}{2}}}\right)\right).
\end{align*}
From Theorem \ref{thm: reg_s_snag_ef_general_nonconvex}, we can see that even in nonconvex cases, acceleration can be beneficial. Indeed, the compression error terms (third and fourth terms) have a better dependence on $\varepsilon$ than S-SGD-EF.
\end{rem}

\section{Numerical Experiments}
In this section, we provide numerical experiments to demonstrate the performances of our methods. \par 
{\bf{Experimental settings}}\ \ \ We conducted standard $L_2$-regularized logistic regression for multi-class classification on publicly available CIFAR 10 dataset\footnotemark. \footnotetext{\url{https://www.cs.toronto.edu/~kriz/cifar.html}.} The regularization parameter was set to be $10^{-4}$. We normalized each channel of images to be mean and standard deviation $0.5$.
We compared our proposed S-SNAG-EF with non-compressed SGD, sparsified SGD without error feedback, top-$k$ sparsified SGD with error feedback and S-SGD-EF. We implemented the all algorithms on pseudo distribution settings in single node. In our experiments, the number of processors $P$ ranged in $\{10, 100\}$ and the compression ratio $k/d$ did in $\{0.01, 0.1\}$. We fairly tuned the all hyper parameters\footnotemark. \footnotetext{For non-compressed SGD, sparsified SGD, top-$k$ SGD with error feedback and S-SGD-EF, we only tuned learning rate $eta$. $eta$ ranged in $\{10^{-i} \mid i \in \{1, \ldots, 6\}$. For S-SNAG-EF, we additionally tunded strong convexity parameter $\mu in \{10^{-i} \mid i \in \{1, \ldots, 4\}$.}We independently ran each experiment four times and report the mean and standard deviation of train and test loss and accuracy against the number of iterations. \par
\begin{figure*}[t]
\begin{subfigmatrix}{4}
\subfigure[$P=10$, $k/d=0.01$ ]{\includegraphics[width=3.4cm]{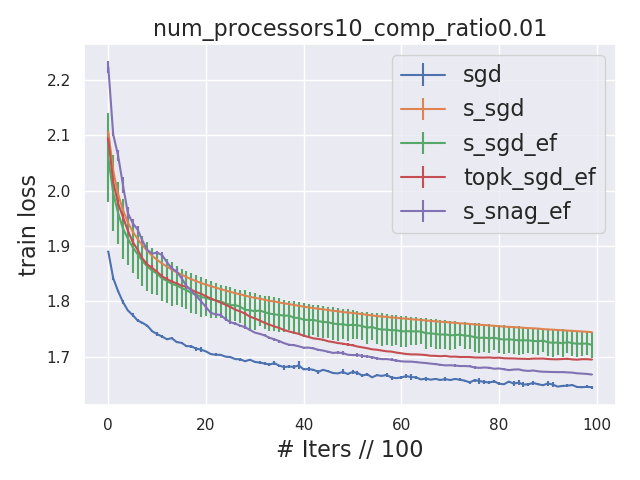}}
\subfigure[$P=10$, $k/d=0.01$ ]{\includegraphics[width=3.4cm]{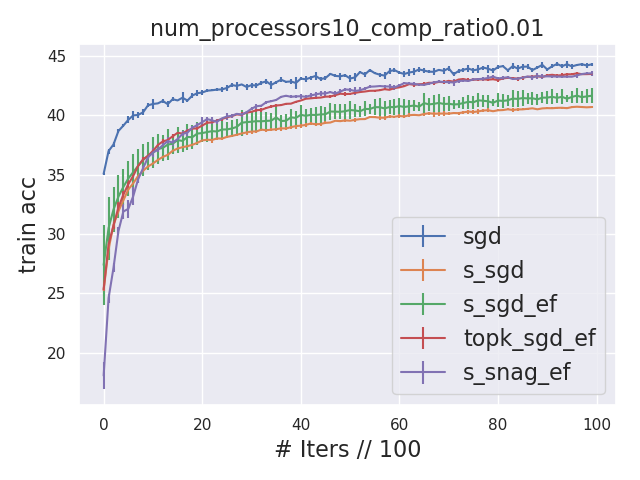}}
\subfigure[$P=10$, $k/d=0.01$]{\includegraphics[width=3.4cm]{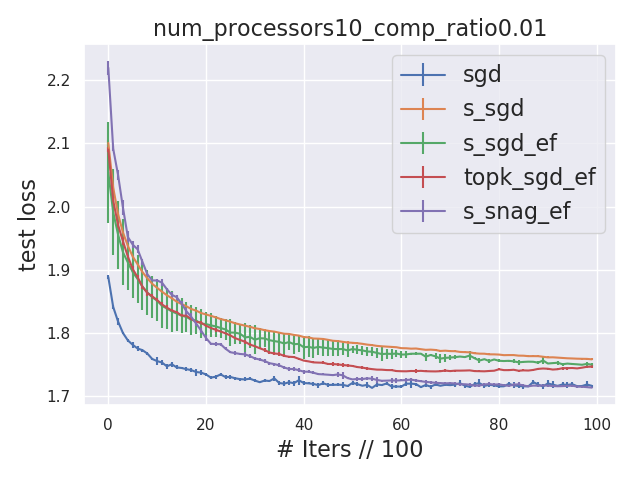}}
\subfigure[$P=10$, $k/d=0.01$
]{\includegraphics[width=3.4cm]{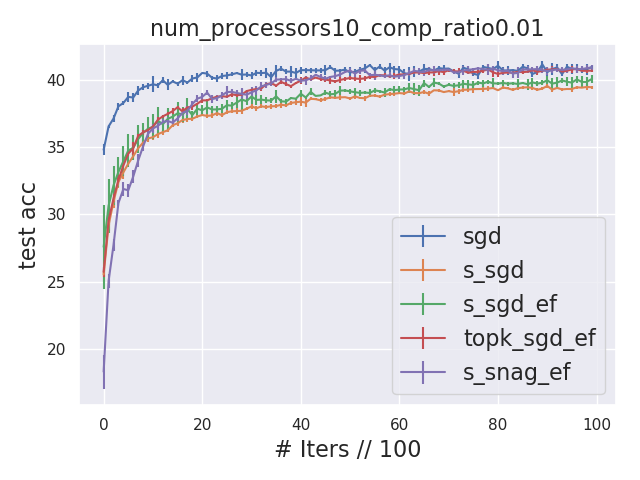}}
\end{subfigmatrix}
\begin{subfigmatrix}{4}
\subfigure[$P=10$, $k/d=0.1$]{\includegraphics[width=3.4cm]{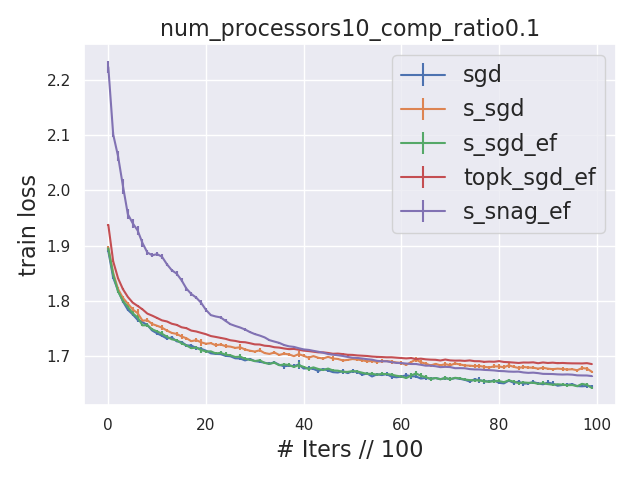}}
\subfigure[$P=10$, $k/d=0.1$]{\includegraphics[width=3.4cm]{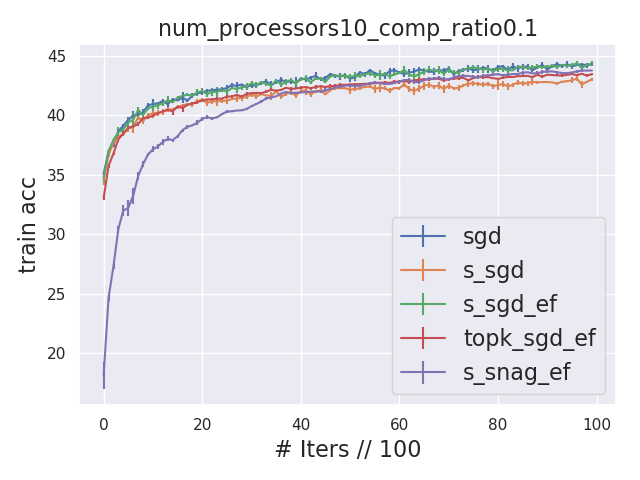}}
\subfigure[$P=10$, $k/d=0.1$]{\includegraphics[width=3.4cm]{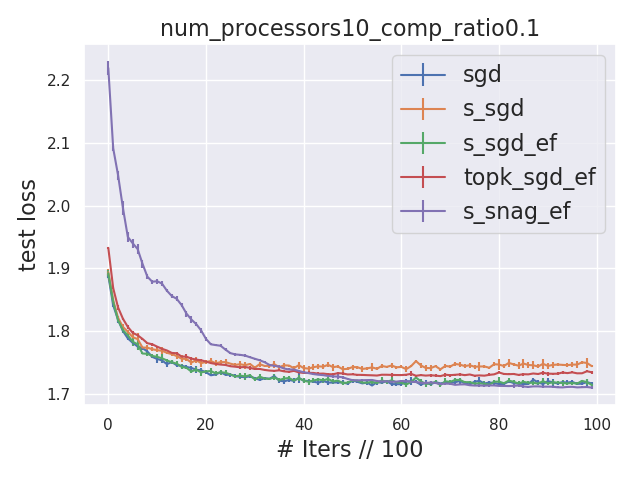}}
\subfigure[$P=10$, $k/d=0.1$
]{\includegraphics[width=3.4cm]{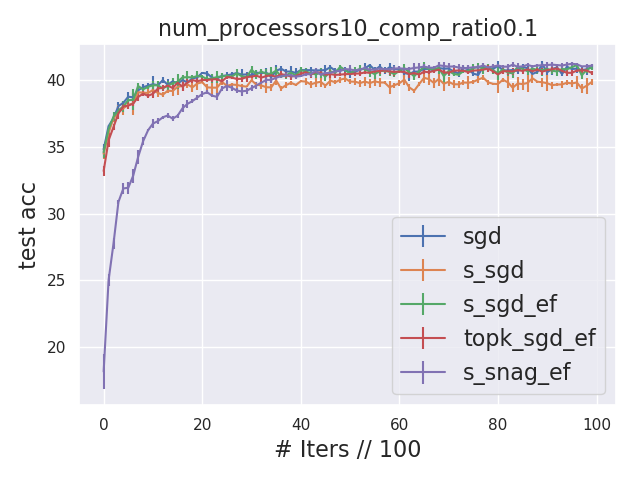}}
\end{subfigmatrix}
\begin{subfigmatrix}{4}
\subfigure[$P=100$, $k/d=0.01$]{\includegraphics[width=3.4cm]{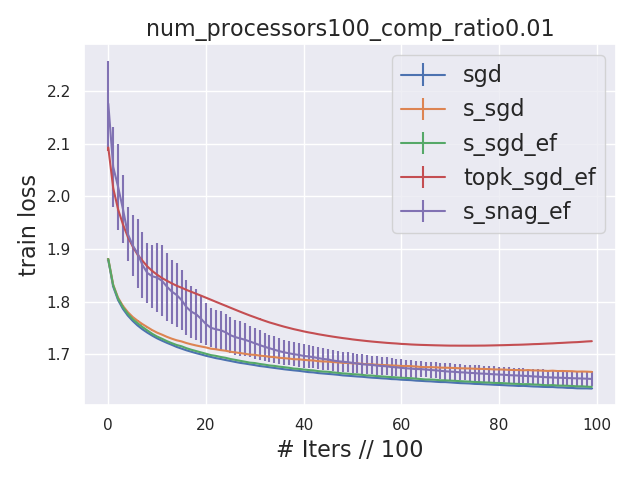}}
\subfigure[$P=100$, $k/d=0.01$]{\includegraphics[width=3.4cm]{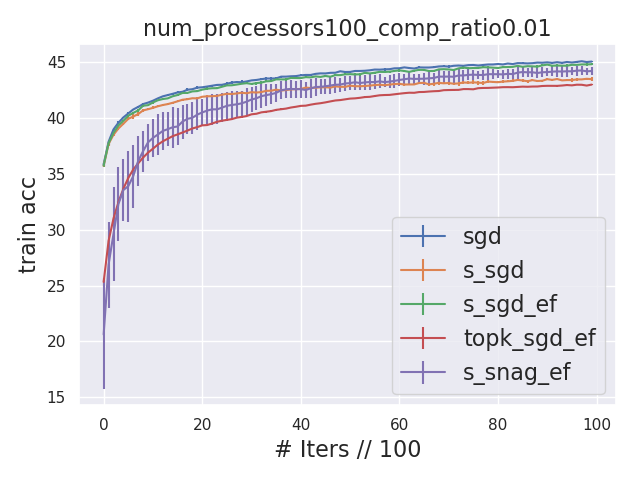}}
\subfigure[$P=100$, $k/d=0.01$]{\includegraphics[width=3.4cm]{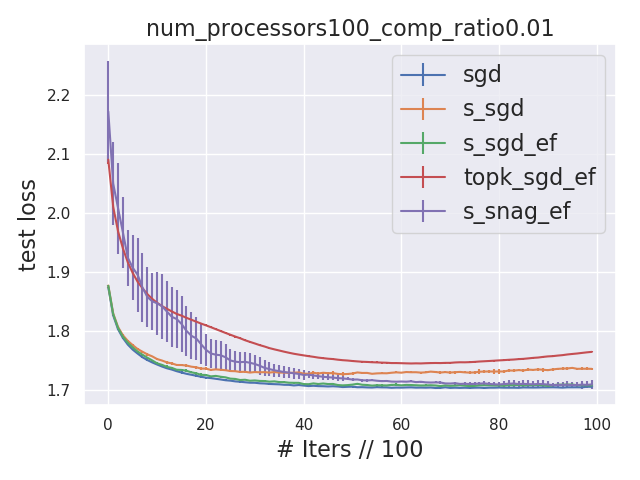}}
\subfigure[$P=100$, $k/d=0.01$
]{\includegraphics[width=3.4cm]{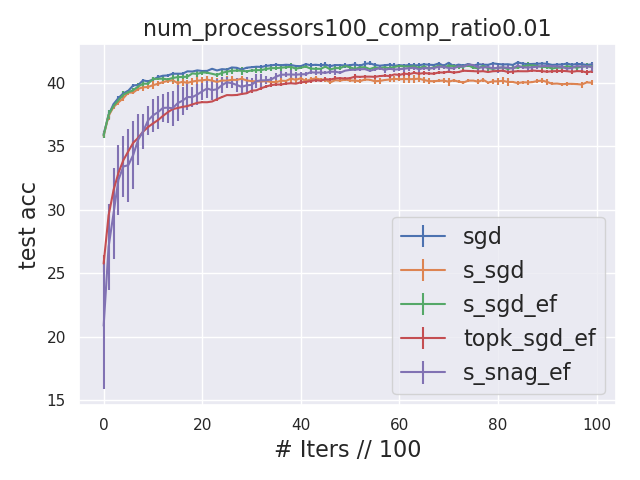}}
\end{subfigmatrix}
\begin{subfigmatrix}{4}
\subfigure[$P=100$, $k/d=0.1$]{\includegraphics[width=3.4cm]{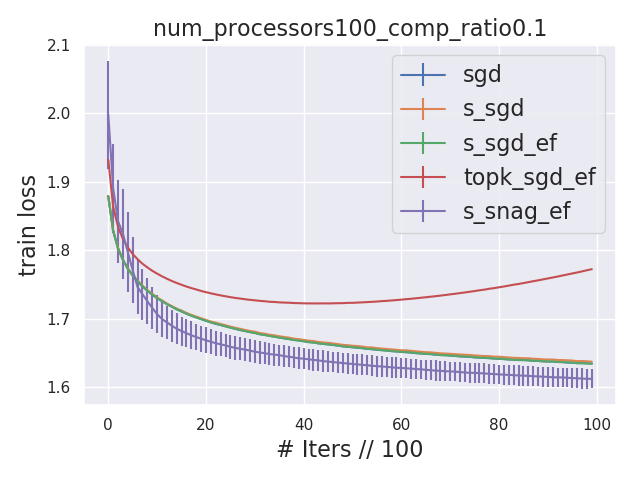}}
\subfigure[$P=100$, $k/d=0.1$]{\includegraphics[width=3.4cm]{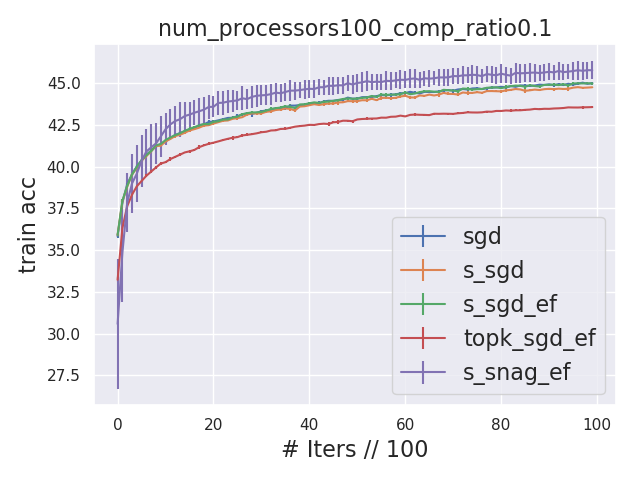}}
\subfigure[$P=100$, $k/d=0.1$]{\includegraphics[width=3.4cm]{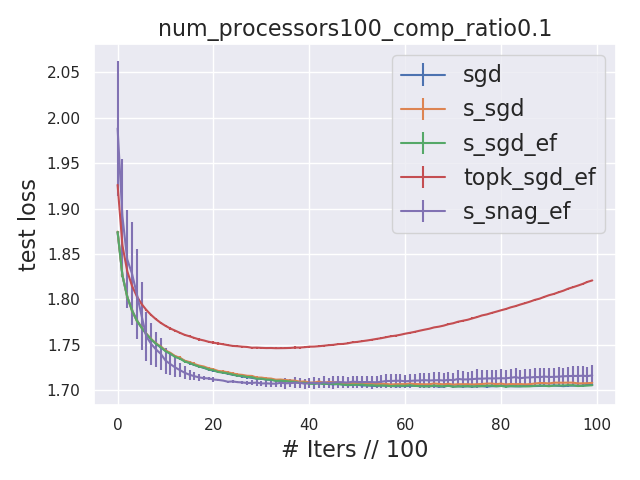}}
\subfigure[$P=100$, $k/d=0.1$]{\includegraphics[width=3.4cm]{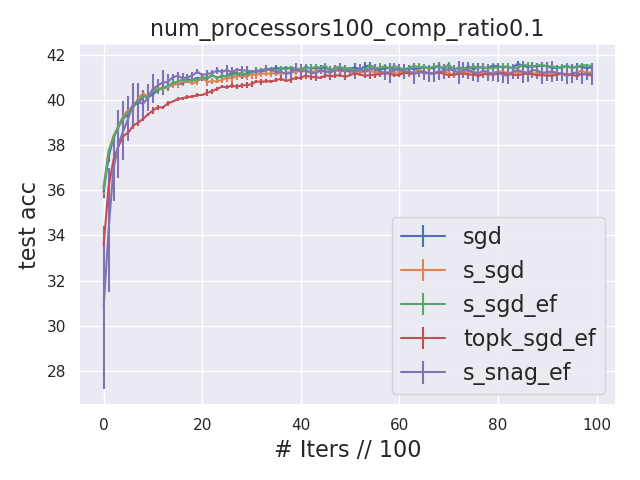}}
\end{subfigmatrix}
\caption{Comparisons of our methods with existing methods on $L_2$-regularized logistic regression tasks on CIFAR 10 for different number of processors $P$ and gradient compression ratio $k/d$. The first, second, third and last columns depict the comparison of train loss, train accuracy, test loss and test accuracy respectively on each $P$ and $k/d$ setting. }
\label{fig: comparison}
\end{figure*}
{\bf{Results}}\ \ \ Figure \ref{fig: comparison} shows the comparisons of our proposed S-SNAG-EF with previous methods and S-SGD-EF. When the cases $(P, k/d)=(10, 0.01), (100, 0.1)$, S-SNAG-EF significantly outperformed the other method except non-compressed SGD. When the cases $(P, k/d)=(10, 0.1), (100, 0.01)$, S-SGD-EF showed the best performances except vanilla SGD. The convergence of S-SNAG-EF was initially slow and this is perhaps a reason why S-SNAG-EF was outperformed by S-SNAG-EF in some cases. The performances of Top-$k$ SGD-EF were unstable and did not converge particularly for large $P$ . 
\section{Conclusion}
In this paper, we considered an accelerated sparsified SGD with error feedback (S-SNAG-EF) in parallel computing settings. We gave theoretical analysis of S-SNAG-EF and showed that our proposed algorithm achieves (i) asymptotical linear speed up with respect to the number of nodes; (ii) lower communication cost for maintaining the rate of vanilla SGD than non-accelerated methods thanks to Nesterov's acceleration. We also gave better analysis of non-accelerated S-SGD-EF than previous work by fully utilizing the unbiasedness of sparsification. In numerical experiments, we compared our methods with several previous methods and our methods showed comparable or better performances.


\section*{Acknowledgement}
TS was partially supported by JSPS KAKENHI (18K19793, 18H03201, and 20H00576), Japan DigitalDesign, and JST CREST.

\bibliography{acc_s_sgd_ef}
\bibliographystyle{abbrvnat}

\appendix 

\section{Analysis of S-SGD-EF}
\label{sec: analysis_s_sgd_ef}
\subsection{Analysis of $\mathbb{E}\|m_t\|^2$}
\begin{lem}\label{lem: s_sgd_ef_cum_error_equality}
\begin{align*}
    \mathbb{E}\|m_t\|^2 = (1-\gamma)^2\|m_{t-1}\|^2 + \eta_t^2\frac{1}{P^2}\sum_{p=1}^P\mathbb{E}\|\nabla f_{i_{t, p}}(x_{t-1}) + (\gamma/\eta_t) m_{t-1, p} - \bar g_{t, p}\|^2.,
\end{align*}
where the expectations are taken with respect to $J_{t, 1}, \ldots, J_{t, P} (\subset [d])$, which are the random choices of the coordinates for constructing $\bar g_{t, p}$ conditioned on $\{i_{t', p}\mid t' \in [t], p \in [P]\}$.
\end{lem}

\begin{proof}
First note that $m_{t, p} = m_{t-1, p} + \eta_t(\nabla f_{i_{t, p}}(x_{t-1}) - \bar g_{t, p}) = \sum_{t'=1}^t \eta_{t'}(\nabla f_{i_{t, p}}(x_{t-1}) - \bar g_{t, p})$ and  $m_t = m_{t-1} + \frac{1}{P}\sum_{p=1}^P \eta_{t'}(\nabla f_{i_{t, p}}(x_{t-1}) - \bar g_{t, p}) = \sum_{t'=1}^t \frac{1}{P}\sum_{p=1}^P \eta_{t'}(\nabla f_{i_{t', p}}(x_{t'-1}) - \bar g_{t', p}) = \frac{1}{P}\sum_{p=1}^Pm_{t, p}$. Since $x_{t} = x_{t-1} - \eta_t\bar g_t $, where $\bar g_t = \frac{1}{P}\sum_{p=1}^P \bar g_{t, p}$  and $\widetilde{x}_t = \widetilde{x}_{t-1} - \eta_t \frac{1}{P}\sum_{p=1}^P\nabla f_{i_{t, p}}(x_{t-1})$, we have 
\begin{align*}
    \mathbb{E}\|m_t\|^2 =&\ \mathbb{E}\|x_{t-1} - \widetilde{x}_{t-1}\|^2 \\
    =&\ \mathbb{E}\left\|m_{t-1} + \eta_t \frac{1}{P}\sum_{p=1}^P(\nabla f_{i_{t, p}}(x_{t-1}) - \bar g_{t, p} )\right\|^2 \\
    =&\ \mathbb{E}\left\|(1-\gamma)m_{t-1} + \eta_t\frac{1}{P}\sum_{p=1}^P\left(\nabla f_{i_{t, p}}(x_{t-1}) + (\gamma/\eta_t) m_{t-1, p} - \bar g_{t, p}\right)\right\|^2.
\end{align*}
Here the expectations are taken with respect to $J_{t, 1}, \ldots, J_{t, P} (\subset [d])$, which are the random choices of the coordinates for constructing $\bar g_{t, p}$ conditioned on $\{i_{t', p}\mid t' \in [t], p \in [P]\}$. Since each $\bar g_{t, p}$ is an independent unbiased estimator of $\nabla f_{i_{t, p}}(x_{t-1}) + (\gamma_t/\eta_t) m_{t-1, p}$ for $p \in [P]$, we have 
\begin{align*}
    \mathbb{E}\|m_t\|^2 
    =&\ \mathbb{E}\left\|(1-\gamma)m_{t-1} + \eta_t\frac{1}{P}\sum_{p=1}^P(\nabla f_{i_{t, p}}(x_{t-1}) + (\gamma/\eta_t) m_{t-1, p} - \bar g_{t, p})\right\|^2 \\
    =&\ (1-\gamma)^2 \|m_{t-1}\|^2 + \eta_t^2\mathbb{E}\left\|\frac{1}{P}\sum_{p=1}^P(\nabla f_{i_{t, p}}(x_{t-1}) + (\gamma/\eta_t) m_{t-1, p} - \bar g_{t, p})\right\|^2 \\
    =&\ (1-\gamma)^2\|m_{t-1}\|^2 + \eta_t^2\frac{1}{P^2}\sum_{p=1}^P\mathbb{E}\|\nabla f_{i_{t, p}}(x_{t-1}) + (\gamma/\eta_t) m_{t-1, p} - \bar g_{t, p}\|^2.
\end{align*}
The last equality is from the independence of $\bar g_{t, 1}, \ldots, \bar g_{t, P}$.
\end{proof}

Now we need to bound the variance term $\mathbb{E}\|\nabla f_{i_{t, p}}(x_{t-1}) + (\gamma/\eta_t) m_{t-1, p} - \bar g_{t, p}\|^2$. 

\begin{lem}\label{lem: s_sgd_ef_variance_bound}
For $p \in [P]$,
\begin{align*}
    &\mathbb{E}\|\nabla f_{i_{t, p}}(x_{t-1}) + (\gamma/\eta_t) m_{t-1, p} - \bar g_{t, p}\|^2 \\
    \leq&\ \Theta\left(\frac{d}{k}\right)(\mathcal{V} + \|\nabla F(x_{t-1})\|^2) + \Theta\left(\frac{d\gamma^2}{k\eta_t^2}\right)\|m_{t-1, p}\|^2 
\end{align*}
\end{lem}
\begin{proof}
Remember that 
$$(\bar g_{t, p})_j 
= \begin{cases}
    \frac{d}{k} (\nabla_j f_{i_{t, p}}(x_{t-1}) + (\gamma/\eta_t) (m_{t-1, p})_j) & (j \in J_{t, p}) \\
    0 & (otherwise)
\end{cases}, $$
where $J_{t, p} = \left\{j_{t, p}^{(1)}, \ldots, j_{t, p}^{(k)}\right\}$ and each $j_{t, p}^{(k)}$ is i.i.d. to the uniform distribution on $[d]$. Since $\left(j_{t, p}^{(k)}\right)_{k=1}^k$ are i.i.d., we have
\begin{align*}
    &\mathbb{E}\|\nabla f_{i_{t, p}}(x_{t-1}) + (\gamma/\eta_t) m_{t-1, p} - \bar g_{t, p}\|^2 \\
    \leq&\ \sum_{j=1}^d  \left\{\frac{k}{d}\left(\left(\frac{d}{k}-1\right)^2 (\nabla_j f_{i_{t, p}}(x_{t-1}) + (\gamma/\eta_t) (m_{t-1, p})_j)^2\right) + \left( 1 - \frac{k}{d}\right)(\nabla_j f_{i_{t, p}}(x_{t-1}) + (\gamma/\eta_t) (m_{t-1, p})_j)^2\right\} \\
    \leq&\ \left(1 + \frac{d}{k}\right)\|\nabla f_{i_{t, p}}(x_{t-1}) + (\gamma/\eta_t) m_{t-1, p}\|^2 \\
    \leq&\ \frac{4d}{k}\|\nabla f_{i_{t, p}}(x_{t-1})\|^2 + \frac{4d\gamma^2}{k\eta_t^2}\|m_{t-1, p}\|^2.
\end{align*}
\end{proof}

\begin{lem}\label{lem: s_sgd_ef_cum_error_expectaion}
For $t \in [T]$ and $p_1 \neq p_2 \in [P]$, 
$$\mathbb{E}\langle m_{t, p_1}, m_{t, p_2}\rangle = 0.$$
Here the expectations are taken with respect to the all random variables.
\end{lem}
\begin{proof}
Suppose that $\{i_{t, p_1}, i_{t, p_2}\}$ and $\{i_{t', p'} \mid t' \in [t-1], p' \in [P] \}$ is given. Since $\mathbb{E}_{J_{t, p_1}}[m_{t, p_1}] = (1-\gamma)m_{t-1, p_1}$, $\mathbb{E}_{J_{t, p_1}}[m_{t, p_2}] = (1-\gamma)m_{t-1, p_2}$ and $J_{t, p_1} and J_{t, p_2}$ are independent, we have
$$\mathbb{E}_{J_{t, p_1}, J_{t, p_2}}\langle m_{t, p_1}, m_{t, p_2}\rangle = (1-\gamma)^2\langle m_{t-1, p_1}, m_{t-1, p_2}\rangle.$$
This implies that 
$$\mathbb{E}\langle m_{t, p_1}, m_{t, p_2}\rangle = (1-\gamma)^2\mathbb{E}\langle m_{t-1, p_1}, m_{t-1, p_2}\rangle,$$
where the expectations are taken with respect to the all random variables. Using this equality recursively, we obtain
$$\mathbb{E}\langle m_{t, p_1}, m_{t, p_2}\rangle = (1-\gamma)^{2(t-1)}\mathbb{E}\langle m_{0, p_1}, m_{0, p_2}\rangle = 0.$$
Here the last equality holds because $m_{0, p} = 0$ for $p \in [P]$.
\end{proof}

\begin{prop}\label{prop: s_sgd_ef_cum_error_recursive_bound}
Let $\gamma = \Theta(k/d)$ be sufficiently small. Then it follows that
\begin{align*}
    \mathbb{E}\|m_t\|^2 \leq (1-\gamma)\|m_{t-1}\|^2 + \Theta\left(\frac{\eta_t^2d}{kP}(\mathcal{V} + \|\nabla F(x_{t-1})\|^2)\right).
\end{align*}
\end{prop}
\begin{proof}
First observe that from Lemma \ref{lem: s_sgd_ef_cum_error_expectaion}, we have
$$\frac{1}{P^2}\sum_{p=1}^P \mathbb{E}\|m_{t-1, p}\|^2 = \mathbb{E}\|m_{t-1}\|^2.$$
Using this fact and combining Lemma \ref{lem: s_sgd_ef_cum_error_equality} with Lemma \ref{lem: s_sgd_ef_variance_bound} give
\begin{align*}
    \mathbb{E}\|m_t\|^2 
    \leq \left((1 - \gamma)^2 + \Theta\left(\frac{d\gamma^2}{k}\right)\right) \|m_{t-1}\|^2 + \Theta\left(\frac{\eta_t^2d}{kP}(\mathcal{V} + \|\nabla F(x_{t-1})\|^2)\right).
\end{align*}
It is easily seen that choosing appropriately small $\gamma = \Theta(k/d)$ is sufficient for ensuring $(1-\gamma)^2 + \Theta(d\gamma^2/k) \leq 1 - \gamma$.
\end{proof}

\begin{prop}\label{prop: s_sgd_ef_cum_error_bound}
Suppose that Assumptions \ref{assump: bounded_variance} holds. Let $\gamma = \Theta(k/d)$ be sufficiently small. Then S-SGD-EF satisfies 
$$\mathbb{E}\|m_t\|^2 \leq \Theta\left(\sum_{t'=1}^t\frac{\eta_t^2 d}{kP} (1-\gamma)^{t - t'}(\mathcal{V} + \mathbb{E}\|\nabla F(x_{t'-1})\|^2)\right).$$
\end{prop}
\begin{rem}
Importantly, the expected accumulated compression error $\mathbb{E}\|m_t\|^2$ is scaled to $1/P$, i.e., linearly scaled with respect to the number of nodes. 
\end{rem}
\begin{pr_s_sgd_ef_cum_error_bound_prop}
The statement is a direct consequence of Proposition \ref{prop: s_sgd_ef_cum_error_recursive_bound}. \qed
\end{pr_s_sgd_ef_cum_error_bound_prop}

\subsection{Analysis for Convex Cases}
\begin{prop}[Strongly convex]\label{prop: s_sgd_ef_convex_bound}
Suppose that Assumptions \ref{assump: sol_existence}, \ref{assump: smoothness}, \ref{assump: bounded_variance} and \ref{assump: strong_convexity} hold. Let $\eta_t = \eta \leq 1/(8L)$.
Then S-SGD-EF satisfies
\begin{align*}
    &\mathbb{E}[F(x_{\mathrm{out}})-F(x_*)] \\
    \leq&\ \Theta\left(\frac{1}{\eta}(1 - \eta \mu)^T\|x_0 - x_*\|^2 + \frac{\eta \mathcal{V}}{P} + \frac{\sum_{t=1}^T(1 - \eta \mu)^{T-t}\left(L\mathbb{E}\|m_{t-1}\|^2 -  \mathbb{E}\|\nabla F(x_{t-1})\|^2/L\right)}{\sum_{t=1}^T (1 - \eta \mu)^{T-t}}\right),
\end{align*}
where $x_{\mathrm{out}} = x_{\hat t-1}$ and $\hat t \sim [T]$ according to  $\left\{(1 - \eta \mu)^{-t}/(\sum_{t=1}^T (1 - \eta \mu)^{-t})\right\}_{t=1}^T$.
\end{prop}
\begin{pr_s_sgd_ef_convex_bound_prop}
Let $\widetilde{x}_t = \widetilde{x}_{t-1} - \eta_t (1/P)\left(\sum_{p=1}^P\nabla f_{i_{t, p}}(x_{t-1})\right)$ and $\widetilde{x}_0 = x_0$.
By the definition of $\widetilde{x}_t$, we have
\begin{align*}
    \|\widetilde{x}_t - x_*\|^2 =&\ \left\|\widetilde{x}_{t-1} - \eta_t \frac{1}{P}\sum_{p=1}^P \nabla f_{i_{t, p}}(x_{t-1}) - x_*\right\|^2 \\
    =&\ \|\widetilde{x}_{t-1} - x_*\|^2 - 2\eta_t \left\langle \frac{1}{P}\sum_{p=1}^P \nabla f_{i_{t, p}}(x_{t-1}), \widetilde{x}_{t-1} - x_* \right\rangle + \eta_t^2\left\|\frac{1}{P}\sum_{p=1}^P\nabla f_{i_{t, p}}(x_{t-1})\right\|^2 \\
    =&\ \|\widetilde{x}_{t-1} - x_*\|^2 -2\eta_t\left\langle \frac{1}{P}\sum_{p=1}^P \nabla f_{i_{t, p}}(x_{t-1}), x_{t-1} - x_*\right\rangle + \eta_t^2\left\|\frac{1}{P}\sum_{p=1}^P\nabla f_{i_{t, p}}(x_{t-1})\right\|^2 \\
    &- 2\eta_t \left\langle \frac{1}{P}\sum_{p=1}^P \nabla f_{i_{t, p}}(x_{t-1}), x_{t-1} - \widetilde{x}_{t-1} \right\rangle.
\end{align*}
Taking expectations with respect to the $t$-th iteration, we get
\begin{align*}
    \mathbb{E}\|\widetilde{x}_t - x_*\|^2 =&\ \|\widetilde{x}_{t-1} - x_*\|^2 - 2\eta_t\langle \nabla F(x_{t-1}), x_{t-1} - x_*\rangle +  \eta_t^2\mathbb{E}\left\|\frac{1}{P}\sum_{p=1}^P\nabla f_{i_{t, p}}(x_{t-1})\right\|^2 \\
    &-2\eta_t \left\langle \nabla F(x_{t-1}), x_{t-1} - \widetilde{x}_{t-1} \right\rangle \\
    \leq&\ \|\widetilde{x}_{t-1} - x_*\|^2 - 2\eta_t\langle \nabla F(x_{t-1}), x_{t-1} - x_*\rangle + \eta_t^2\mathbb{E}\left\|\frac{1}{P}\sum_{p=1}^P\nabla f_{i_{t, p}}(x_{t-1}) - \nabla F(x_{t-1})\right\|^2 \\
    &+ \eta_t\left(\frac{1}{4L} + \eta_t\right)\|\nabla F(x_{t-1})\|^2 + 4\eta_t L\|x_{t-1} - \widetilde{x}_{t-1}\|^2.
\end{align*}
Here the last inequality follows from the unbiasedness of $(1/P)\sum_{p=1}^P\nabla f_{i_{t, p}}(x_{t-1})$ and Cauchy-Schwartz inequality with the arithmetic-geometric mean inequality. Since $F$ is $L$-smooth and $\mu$-strongly convex, we have
$$F(x_{t-1}) + \langle \nabla F(x_{t-1}), x_* - x_{t-1}\rangle + \frac{1}{4L}\|\nabla F(x_{t-1})\|^2 + \frac{\mu}{4}\|x_{t-1} - x_*\|^2 \leq F(x_*), $$
and this implies
$$ -2\eta_t \langle \nabla F(x_{t-1}), x_{t-1} - x_* \rangle \leq -2\eta_t(F(x_{t-1}) - F(x_*)) - \frac{\eta_t}{2L}\|\nabla F(x_{t-1})\|^2 - \frac{\eta_t\mu}{2}\|x_{t-1} - x_*\|^2 .$$
Applying this inequality to the above one, we get
\begin{align*}
    \mathbb{E}\|\widetilde{x}_t - x_*\|^2 \leq&\ \left(1 - \frac{\eta_t\mu}{2}\right)\|\widetilde{x}_{t-1} - x_*\|^2 -2\eta_t(F(x_{t-1}) - F(x_*)) - \left(\frac{\eta_t}{4L} - \eta_t^2\right)\|\nabla F(x_{t-1})\|^2 \\
    &+  \eta_t^2\mathbb{E}\left\|\frac{1}{P}\sum_{p=1}^P\nabla f_{i_{t, p}}(x_{t-1}) - \nabla F(x_{t-1})\right\|^2 + 4\eta_tL\|x_{t-1} - \widetilde{x}_{t-1}\|^2
\end{align*}
Noting that $\mathbb{E}\|(1/P)\sum_{p=1}^P\nabla f_{i_{t, p}}(x_{t-1}) - \nabla F(x_{t-1})\|^2 \leq \mathcal{V}/P$ and $\eta_t \leq 1/(8L)$ from the assumption. Taking expectations with respect to the all random variables, we have
\begin{align*}
    \mathbb{E}\|\widetilde{x}_t - x_*\|^2 \leq&\ \left(1 - \frac{\eta_t\mu}{2}\right)\mathbb{E}\|\widetilde{x}_{t-1} - x_*\|^2 -2\eta_t\mathbb{E}[F(x_{t-1}) - F(x_*)] + \frac{\eta_t^2\mathcal{V}}{P} \\ 
    &+ 2\eta_t L \mathbb{E}\|m_{t-1}\|^2 - \frac{\eta_t}{8L}\mathbb{E}\|\nabla F(x_{t-1})\|^2.
\end{align*}
Here we used $x_{t-1} - \widetilde{x}_{t-1} = m_{t-1}$.

Recursively using the above inequality (with $\eta_t = \eta$) and rearranging the result give
\begin{align*}
    &\mathbb{E}[F(x_{\mathrm{out}})-F(x_*)] \\
    \leq&\  \Theta\left(\frac{(1 - \eta \mu)^T\|x_0 - x_*\|^2}{\eta \sum_{t=1}^T (1 - \eta \mu)^{T-t}} + \frac{\eta \mathcal{V}}{P} + \frac{\sum_{t=1}^T(1 - \eta \mu)^{T-t}\left(L\mathbb{E}\|m_{t-1}\|^2 -  \mathbb{E}\|\nabla F(x_{t-1})\|^2/L\right)}{\sum_{t=1}^T (1 - \eta \mu)^{T-t}}\right) \\ 
    \leq&\ \Theta\left(\frac{1}{\eta}(1 - \eta \mu)^T\|x_0 - x_*\|^2 + \frac{\eta \mathcal{V}}{P} + \frac{\sum_{t=1}^T(1 - \eta \mu)^{T-t}\left(L\mathbb{E}\|m_{t-1}\|^2 -  \mathbb{E}\|\nabla F(x_{t-1})\|^2/L\right)}{\sum_{t=1}^T (1 - \eta \mu)^{T-t}}\right).
\end{align*}
This is the desired result.
\end{pr_s_sgd_ef_convex_bound_prop}

\begin{lem}\label{lem: s_sgd_ef_geometric_sum}
Let $0 < r_1, r_2 < 1$ and $2r_2 \leq r_1$. Then for any non-negative sequence $\{c_t\}_{t=1}^{\infty}$, 
\begin{align*}
    \sum_{t=1}^T(1 - r_2)^{T-t}\sum_{t'=1}^t (1 - r_1)^{t-t'}c_{t'} \leq \frac{2}{r_1}\sum_{t=1}^T(1 - r_2)^{T-t}c_t.
\end{align*}
\end{lem}
\begin{proof}
\begin{align*}
    \sum_{t=1}^T(1 - r_2)^{T-t}\sum_{t'=1}^t (1 - r_1)^{t-t'}c_{t'} 
    =&\ \sum_{t=1}^T (1 - r_2)^{T-t}\sum_{t'=1}^T \mathbbm{1}_{t' \leq t}(1 - r_1)^{t-t'}c_{t'} \\
    =&\ \sum_{t'=1}^T c_{t'} (1 - r_2)^{T-t'} \sum_{t=t'}^T \left(\frac{1 - r_1}{1 - r_2}\right)^{t - t'} \\
    \leq&\ \sum_{t'=1}^T c_{t'} (1 - r_2)^{T-t'}\frac{1 - \left(\frac{1-r_1}{1-r_2}\right)^{T-t'}}{1 - \frac{1-r_1}{1-r_2}} \\
    \leq&\ \frac{2(1-r_2)}{r_1}\sum_{t=1}^T(1 - r_2)^{T-t}c_t.
\end{align*}
\end{proof}

\begin{thm}[Strongly convex]\label{thm: s_sgd_ef_strongly_convex}
Suppose that Assumptions \ref{assump: sol_existence}, \ref{assump: smoothness}, \ref{assump: bounded_variance} and \ref{assump: strong_convexity} hold. 
Let $\gamma = \Theta(k/d)$ be sufficiently small and $T = \widetilde\Theta(1/(\eta\mu))$ be sufficiently large. Then the iteration complexity $T$ of S-SGD-EF with appropriate $\eta_t = \eta$ for achieving $\mathbb{E}[F(x_{\mathrm{out}}) - F(x_*)] \leq \varepsilon$ is 
\begin{align*}
     \widetilde{O}\left( \frac{L}{\mu} + \frac{\mathcal{V}}{P}\frac{1}{\mu\varepsilon}  +  \frac{d}{k} + \frac{d}{k\sqrt{P}}\left(\frac{L}{\mu} + \frac{\sqrt{L\mathcal{V}}}{\mu\sqrt{\varepsilon}}\right) \right),
\end{align*}
where $x_{\mathrm{out}}$ is defined in Proposition \ref{prop: s_sgd_ef_convex_bound}.
\end{thm}
\begin{rem}
Theorem \ref{thm: s_sgd_ef_strongly_convex} implies that S-SGD-EF asymptotically achieves $\mathcal{V}/(P\mu\varepsilon)$, that is the asymptotic iteration complexity of non-sparsified parallel SGD, because the last compression error term has a dependence on $1/\sqrt{\varepsilon}$ rather than $1/\varepsilon$. Also note that the last term is scaled to $\sqrt{P}$. This is a desirable property for distributed optimization with $P \gg 1$. However, the last term has a factor of $d/k$, which may be large and can dominate the other terms for moderate accuracy $\varepsilon$. Thus, consideration of non-asymptotic behavior is also important particularly for high compression settings.
\end{rem}
\begin{pr_s_sgd_ef_strongly_convex_thm}
Let $\eta_t = \eta = \Theta(1/L \wedge \gamma \sqrt{P}/L \wedge \gamma/\mu
 \wedge P\varepsilon/\mathcal{V} \wedge (L\mathcal{V})^{-1/2}\gamma\sqrt{P\varepsilon})$.
From Proposition \ref{prop: s_sgd_ef_convex_bound}, we have
\begin{align}
    &\mathbb{E}[F(x_{\mathrm{out}})-F(x_*)] \notag \\
    \leq&\ \Theta\left(\frac{1}{\eta}(1 - \eta \mu)^T\|x_0 - x_*\|^2 + \frac{\eta \mathcal{V}}{P} + \frac{\sum_{t=1}^T(1 - \eta \mu)^{T-t}\left(L\mathbb{E}\|m_{t-1}\|^2 -  \mathbb{E}\|\nabla F(x_{t-1})\|^2/L\right)}{\sum_{t=1}^T (1 - \eta \mu)^{T-t}}\right) \notag \\ 
     \leq&\ \Theta\left(\frac{1}{\eta}(1 - \eta \mu)^T\|x_0 - x_*\|^2 + \frac{\eta \mathcal{V}}{P} + \frac{\eta^2 Ld^2\mathcal{V}}{k^2P}\right. \notag \notag \\ 
    &\left.+  \frac{\sum_{t=1}^T \frac{\eta^2Ld}{kP}(1 - \eta \mu)^{T-t}\sum_{t'=1}^t (1-\gamma)^{t - t'}\mathbb{E}\|\nabla F(x_{t'-1})\|^2 - \sum_{t=1}^T\frac{1}{L}(1 - \eta \mu)^{T-t}\mathbb{E}\|\nabla F(x_{t-1})\|^2}{\sum_{t=1}^T (1 - \eta \mu)^{T-t}}\right) \label{ineq: s_sgd_ef_grad_diff}
\end{align}
Since $\eta \leq \Theta(\gamma/\mu)$ be sufficiently small, from Lemma \ref{lem: s_sgd_ef_geometric_sum}, we have
\begin{align*}
    \sum_{t=1}^T \frac{\eta^2Ld}{kP}(1 - \eta \mu)^{T-t}\sum_{t'=1}^t (1-\gamma)^{t - t'}\mathbb{E}\|\nabla F(x_{t'-1})\|^2 \leq \Theta\left(\frac{\eta^2Ld^2}{k^2P}\sum_{t=1}^T (1 - \eta \mu)^{T-t}\mathbb{E}\|\nabla F(x_{t-1})\|^2 \right).
\end{align*}
Also, since $\eta \leq \Theta(k\sqrt{P}/(dL))$ is assumed, The last term in (\ref{ineq: s_sgd_ef_grad_diff}) is negative with appropriate choice of $\eta$. Hence, we obtain
\begin{align*}
    &\mathbb{E}[F(x_{\mathrm{out}})-F(x_*)] \\  
    \leq&\ \Theta\left(\frac{1}{\eta}(1 - \eta \mu)^T\|x_0 - x_*\|^2 + \frac{\eta \mathcal{V}}{P} + \frac{\eta^2 Ld^2\mathcal{V}}{k^2P}\right).
\end{align*}
Let $T = \Theta(1/(\eta\mu)\mathrm{log}(\|x_0 - x_*\|^2/(\eta\varepsilon))$ to be sufficiently large. Then substituting the definition of $\eta$ gives $\mathbb{E}[F(x_{\mathrm{out}}) - F(x_*)] \leq \varepsilon$. \qed
\end{pr_s_sgd_ef_strongly_convex_thm}

\subsection{Analysis for Nonconvex Cases}
\begin{prop}[General nonconvex]\label{prop: s_sgd_ef_nonconvex_bound}
Suppose that Assumptions \ref{assump: sol_existence}, \ref{assump: smoothness} and \ref{assump: bounded_variance} hold. Assume that $\eta_t = \eta \leq 1/(2L)$ for $t \in \mathbb{N}$. Then S-SGD-EF satisfies 
\begin{align*}
\mathbb{E}\|\nabla F(x_{\mathrm{out}})\|^2 \leq \Theta\left(\frac{F(x_{\mathrm{in}}) - F(x_*)}{\eta T} + \frac{\eta L\mathcal{V}}{P} + \frac{L^2}{T} \sum_{t=1}^T\mathbb{E}\|m_{t-1}\|^2 - \frac{1}{T}\sum_{t=1}^T \mathbb{E}\|\nabla F(x_{T-1})\|^2\right), \end{align*}
where $x_{\mathrm{out}} = x_{\hat t-1}$ and $\hat t \sim [T]$ with probability $\{1/T\}_{t=1}^T$.
\end{prop}
\begin{proof}
Let $\widetilde{x}_t = \widetilde{x}_{t-1} - \eta_t (1/P)\left(\sum_{p=1}^P\nabla f_{i_{t, p}}(x_{t-1})\right)$ and $\widetilde{x}_0 = x_0$.
By the $L$-smoothness of $F$, we have
\begin{align*}
    F(\widetilde{x}_t) \leq F(\widetilde{x}_{t-1}) + \langle \nabla F(\widetilde{x}_{t-1}), \widetilde{x}_t - \widetilde{x}_{t-1} \rangle + \frac{L}{2}\|\widetilde{x}_t - \widetilde{x}_{t-1}\|^2. 
\end{align*}
Since $\widetilde{x}_t = \widetilde{x}_{t-1} - \eta_t \left(\frac{1}{P}\sum_{p=1}^P\nabla f_{i_{t, p}}(x_{t-1})\right)$, it follows that
\begin{align*}
    F(\widetilde{x}_t) \leq&\  F(\widetilde{x}_{t-1}) - \eta_t \left\langle \nabla F(\widetilde{x}_{t-1}),  \frac{1}{P}\sum_{p=1}^P\nabla f_{i_{t, p}}(x_{t-1}) \right\rangle + \frac{L}{2}\left\| \frac{1}{P}\sum_{p=1}^P\nabla f_{i_{t, p}}(x_{t-1})\right\|^2 \\
    =&\ F(\widetilde{x}_{t-1}) - \eta_t \left\langle \nabla F(x_{t-1}), \frac{1}{P}\sum_{p=1}^P\nabla f_{i_{t, p}}(x_{t-1}) \right\rangle \\
    &+ \eta_t \left\langle \nabla F(x_{t-1}) - \nabla F(\widetilde{x}_{t-1}), \frac{1}{P}\sum_{p=1}^P\nabla f_{i_{t, p}}(x_{t-1}) \right\rangle
    + \frac{\eta_t^2L}{2}\left\| \frac{1}{P}\sum_{p=1}^P\nabla f_{i_{t, p}}(x_{t-1})\right\|^2.
\end{align*}
Taking expectations with respect to $\{i_{t, 1}, \ldots, i_{t, P}\}$ conditioned on $\{i_{t', p} \mid t' \in [t-1], p \in [P]\}$, we have
\begin{align*}
    \mathbb{E}[F(\widetilde{x}_t)] \leq&\ F(\widetilde{x}_{t-1}) - \eta_t\|\nabla F(x_{t-1})\|^2 + \eta_t \left\langle \nabla F(x_{t-1}) - \nabla F(\widetilde{x}_{t-1}), \nabla F(x_{t-1}) \right\rangle \\
    &+ \frac{\eta_t^2L}{2}\mathbb{E}\left\|\frac{1}{P}\sum_{p=1}^P\nabla f_{i_{t, p}}(x_{t-1})\right\|^2. 
\end{align*}
Using 
$$\left\langle \nabla F(x_{t-1}) - \nabla F(\widetilde{x}_{t-1}), \nabla F(x_{t-1}) \right\rangle \leq \frac{1}{2}\|\nabla F(x_{t-1}) - \nabla F(\widetilde{x}_{t-1})\|^2 + \frac{1}{2}\|\nabla F(x_{t-1})\|^2$$
and
$$\mathbb{E}\left\|\frac{1}{P}\sum_{p=1}^P\nabla f_{i_{t, p}}(x_{t-1})\right\|^2 = \mathbb{E}\left\|\frac{1}{P}\sum_{p=1}^P\nabla f_{i_{t, p}}(x_{t-1}) - \nabla F(x_{t-1})\right\|^2 + \|\nabla F(x_{t-1})\|^2, $$
we get
\begin{align*}
    \mathbb{E}[F(\widetilde{x}_t)] \leq&\ F(\widetilde{x}_{t-1}) - \frac{\eta_t}{2}\left(1 - \eta_t L\right)\|\nabla F(x_{t-1})\|^2 + \frac{\eta_t^2L}{2}\mathbb{E}\left\|\frac{1}{P}\sum_{p=1}^P\nabla f_{i_{t, p}}(x_{t-1}) - \nabla F(x_{t-1})\right\|^2 \\
    &+ \frac{\eta_t}{2}\| \nabla F(x_{t-1}) - \nabla F(\widetilde{x}_{t-1})\|^2  \\
    \leq&\ F(\widetilde{x}_{t-1}) - \frac{\eta_t}{2}\left(1 - \eta_t L\right)\|\nabla F(x_{t-1})\|^2 + \frac{\eta_t^2L}{2}\mathbb{E}\left\|\frac{1}{P}\sum_{p=1}^P\nabla f_{i_{t, p}}(x_{t-1}) - \nabla F(x_{t-1})\right\|^2 \\
    &+ \frac{\eta_tL^2}{2}\|x_{t-1} - \widetilde{x}_{t-1}\|^2 \\
    \leq&\ F(\widetilde{x}_{t-1}) - \frac{\eta_t}{2}\left(1 - \eta_t L\right)\|\nabla F(x_{t-1})\|^2 + \frac{\eta_t^2L\mathcal{V}}{2P}
    + \frac{\eta_tL^2}{2}\|x_{t-1} - \widetilde{x}_{t-1}\|^2 \\
    \leq&\ F(\widetilde{x}_{t-1}) - \frac{\eta_t}{4}\|\nabla F(x_{t-1})\|^2 + \frac{\eta_t^2L\mathcal{V}}{2P}
    + \frac{\eta_tL^2}{2}\|m_{t-1}\|^2
\end{align*}
Here the second inequality follows from $L$-smoothness of $F$. The third inequality holds because $\mathbb{E}\left\|(1/P)\sum_{p=1}^P\nabla f_{i_{t, p}}(x_{t-1}) - \nabla F(x_{t-1})\right\|^2 \leq \frac{\mathcal{V}}{P}$. The last inequality is due to the fact that $m_{t-1} = x_{t-1} - \widetilde{x}_{t-1}$ and $\eta_t \leq \frac{1}{2L}$.  Rearranging the inequality and taking expectations with respect to the history of all random variables yield
$$\frac{\eta_t}{4}\mathbb{E}\|\nabla F(x_{t-1})\|^2 \leq \mathbb{E}[F(\widetilde{x}_{t-1}) - F(\widetilde{x}_t)] + \frac{\eta_t^2L \mathcal{V}}{2P} + \frac{\eta_t L^2}{2}\mathbb{E}\|m_{t-1}\|^2 -  \frac{\eta_t}{4}\|\nabla F(x_{t-1})\|^2 .$$

Finally, since $\widetilde{x}_0 = x_0$ and $F(\widetilde{x}_{T}) \geq F(x_*)$,  summing this inequality from $t=1$ to $t=T$ and dividing the result by $\sum_{t=1}^T\eta_t$ give the desired results. 
\end{proof}

For nonconvex objectives, we can derive the following proposition. 
\begin{prop}[General nonconvex]\label{prop: s_sgd_ef_nonconvex_bound}
Suppose that Assumptions \ref{assump: sol_existence}, \ref{assump: smoothness} and \ref{assump: bounded_variance} hold. Assume that $\eta_t = \eta \leq 1/(2L)$. Then S-SGD-EF satisfies 
\begin{align*}
\mathbb{E}\|\nabla F(x_{\mathrm{out}})\|^2 \leq \Theta\left(\frac{F(x_{\mathrm{in}}) - F(x_*)}{\eta T} + \frac{\eta L\mathcal{V}}{P} + \frac{L^2}{T} \sum_{t=1}^T\mathbb{E}\|m_{t-1}\|^2 - \frac{1}{T}\sum_{t=1}^T \mathbb{E}\|\nabla F(x_{T-1})\|^2\right), \end{align*}
where $x_{\mathrm{out}} = x_{\hat t-1}$ and $\hat t \sim [T]$ with probability $\{1/T\}_{t=1}^T$.
\end{prop}
Combining Proposition \ref{prop: s_sgd_ef_nonconvex_bound} with Proposition \ref{prop: s_sgd_ef_cum_error_bound} yields the following theorem.
\begin{thm}[General nonconvex]\label{thm: s_sgd_ef_general_nonconvex}
Suppose that Assumptions \ref{assump: sol_existence}, \ref{assump: smoothness} and \ref{assump: bounded_variance} hold. 
Let $\gamma$ be the same one in Theorem \ref{thm: s_sgd_ef_strongly_convex}. Then the iteration complexity $T$ of S-SGD-EF with appropriate $\eta_t = \eta$ to acheive $\mathbb{E}\|\nabla F(x_{\mathrm{out}})\|^2 \leq \varepsilon$ is 
\begin{align*}
     O\left( \frac{L\Delta}{\varepsilon} + \frac{\mathcal{V}}{P}\frac{L\Delta}{\varepsilon^2} + \frac{d}{k} + \frac{d}{k\sqrt{P}} \left(\frac{L\Delta}{\varepsilon} + \frac{L\sqrt{\mathcal{V}}\Delta}{\varepsilon^{\frac{3}{2}}}\right)\right),
\end{align*}
where $\Delta = F(x_\mathrm{in}) - F(x_*)$ and $x_{\mathrm{out}}$ is defined in Proposition \ref{prop: s_sgd_ef_nonconvex_bound}.
\end{thm}
Similar to convex cases, S-SGD-EF asymptotically achieves the same rate as non-sparsified SGD.

\begin{pr_s_sgd_ef_general_nonconvex_thm}
Let $\eta_t = \eta = \Theta(1/L \wedge \gamma \sqrt{P}/L \wedge P\varepsilon/(L\mathcal{V}) \wedge L\mathcal{V}^{-1/2}\gamma\sqrt{P\varepsilon})$ be the same one defined in the proof of Theorem \ref{thm: s_sgd_ef_strongly_convex}.
From Proposition \ref{prop: s_sgd_ef_nonconvex_bound}, similar to the proof of Theorem \ref{thm: s_sgd_ef_strongly_convex}, we have
\begin{align*}
    \mathbb{E}\|\nabla F(x_{\mathrm{out}})\|^2 \leq \Theta\left(\frac{F(x_{0}) - F(x_*)}{\eta T} + \frac{\eta L\mathcal{V}}{P} + \frac{\eta^2L^2d^2\mathcal{V}}{k^2P}\right).
\end{align*}
Set $T = \Theta(\Delta/(\eta\varepsilon))$. Substituting the definition of $\eta$, it can be easily seen that $\mathbb{E}\|\nabla F(x_{\mathrm{out}})\|^2 \leq \varepsilon$ and we obtain the desired result. \qed
\end{pr_s_sgd_ef_general_nonconvex_thm}

\section{Analysis of S-SNAG-EF}\label{sec: analysis_s_snag_ef}
\subsection{Analysis of $\mathbb{E}\|m_t\|^2$}
Remind that for $t \in [T]$,
$$\begin{cases}
    y_t &= x_{t-1} - \eta_t \frac{1}{P}\sum_{p=1}^P \bar g_{t, p}^{(y)}, \\
    z_t &= (1 - \beta_t)z_{t-1} + \beta_t x_{t-1} - \lambda_t \frac{1}{P}\sum_{p=1}^P \bar g_{t, p}^{(z)}, \\
    x_t &= (1-\alpha_{t})y_{t} + \alpha_{t} z_{t}. \\
\end{cases}$$
Let
$$\begin{cases}
    \widetilde y_t& = \widetilde x_{t-1} - \eta_t \frac{1}{P}\sum_{p=1}^P \nabla f_{i_{t, p}}(x_{t-1}), \\
    \widetilde z_t& = (1 - \beta_t)\widetilde z_{t-1} + \beta_t \widetilde{x}_{t-1} - \lambda_t \frac{1}{P}\sum_{p=1}^P \nabla f_{i_{t, p}}(x_{t-1}), \\
    \widetilde{x}_{t}& = (1-\alpha_{t})\widetilde{y}_{t} + \alpha_{t} \widetilde{z}_{t}, \\
\end{cases}$$
where $\widetilde y_0 = y_0$, $\widetilde z_0 = z_0$ and $\widetilde x_0 = x_0$
\begin{lem}
\label{lem: s_snag_ef_cum_error_relations}
For $t \in [T]$, let $m_t = (1/P)\sum_{p=1}^P m_{t, p}$, $m_t^{(y)} = (1/P)\sum_{p=1}^P m_{t, p}^{(y)}$ and $m_t^{(z)} = (1/P)\sum_{p=1}^P m_{t, p}^{(z)}$. Then, it holds that 
\begin{align*}
\begin{cases} m_t^{(y)} = y_t - \widetilde{y}_t, \\ m_t^{(z)} = z_t - \widetilde{z}_t, \\ m_t = x_t - \widetilde{x}_t. 
\end{cases}
\end{align*}
\end{lem}
\begin{proof}
We show the claim by mathematical induction. For $t = 1$, $m_1^{(y)} = m_0 + \eta_1 (1/P)\sum_{p=1}^P (\nabla f_{i, p}(x_0) - \bar g_{1, p}^{(y)}) = y_1 - \widetilde{y}_1$. Similarly, $m_1^{(z)} = (1 - \beta_1)m_0^{(z)} + \beta_1 m_0 + \lambda_1(1/P)\sum_{p=1}^P (\nabla f_{i, p}(x_0) - \bar g_{1, p}^{(z)}) = z_1 - \widetilde z_1$. Also, $m_1 = (1 - \alpha_1)m_1^{(y)} + \alpha_1 m_1^{(z)} = (1 - \alpha_1)(y_1 - \widetilde y_1) + \alpha_1 (z_1 - \widetilde z_1) = x_1 - \widetilde{x}_1$. Hence the statements hold for $t=1$. Suppose that the statements hold for $t = t-1$. Then, $m_{t}^{(y)} = m_{t-1} + \eta_t (1/P)\sum_{p=1}^P (\nabla f_{i, p}(x_{t-1}) - \bar g_{t, p}^{(y)}) = x_{t-1} - \widetilde{x}_{t-1} + \eta_t (1/P)\sum_{p=1}^P (\nabla f_{i, p}(x_{t-1}) - \bar g_{t, p}^{(y)}) = y_t - \widetilde{y}_t$. Similarly, $m_t^{(z)} = (1 - \beta_t)m_{t-1}^{(z)} + \beta_t m_{t-1}  + \lambda_t(1/P)\sum_{p=1}^P (\nabla f_{i, p}(x_{t-1}) - \bar g_{t, p}^{(z)})= z_t - \widetilde z_t$. Also, 
$m_t = (1 - \alpha_t)m_t^{(y)} + \alpha_t m_t^{(z)} = (1-\alpha_t)(y_t - \widetilde{y}_t) + \alpha_t(z_t - \widetilde{z}_t) = x_t - \widetilde{x}_t$. Therefore the statements hold for any $t \in [T]$.
\end{proof}

\begin{lem}
\label{lem: s_snag_ef_cum_error_expectation}
For $t \in [T]$ and $p_1 \neq p_2 \in [P]$, 
$$\mathbb{E}\langle m_{t, p_1}, m_{t, p_2}\rangle = 0.$$
Here the expectations are taken with respect to the all random variables.
\end{lem}
\begin{proof}
We show that $\mathbb{E}\langle m_{t, p_1}^{(y)}, m_{t, p_2}^{(y)}\rangle = 0$, $\mathbb{E}\langle m_{t, p_1}^{(y)}, m_{t, p_2}^{(z)}\rangle = 0$ and $\mathbb{E}\langle m_{t, p_1}^{(z)}, m_{t, p_2}^{(z)}\rangle = 0$ by mathematical induction. For $t=1$, the statements are trivial because of the independence of the random choices of non-sparsified coordinates and $\mathbb{E}[\bar g_{1, p}^{(y)}] = \mathbb{E}[\bar g_{1, p}^{(z)}] = \nabla f_{i, p}(x_0)$ for $p \in [P]$. Suppose that $\mathbb{E}\langle m_{t-1, p_1}^{(y)}, m_{t-1, p_2}^{(y)}\rangle = \mathbb{E}\langle m_{t-1, p_1}^{(y)}, m_{t-1, p_2}^{(z)}\rangle = \mathbb{E}\langle m_{t-1, p_1}^{(z)}, m_{t-1, p_2}^{(z)}\rangle = 0$ hold. 
$\mathbb{E}[m_{t, p_1}^{(y)} | t-1 ] = (1-\gamma)m_{t-1, p_1} = (1-\gamma)((1-\alpha_t) m_{t-1, p_1}^{(y)} + \alpha_t m_{t-1, p_1}^{(z)})$ and hence $\mathbb{E}\langle m_{t, p_1}^{(y)}, m_{t, p_2}^{(y)}\rangle = 0$ by the inductive assumptions. Similarly, we have $\mathbb{E}\langle m_{t, p_1}^{(y)}, m_{t, p_2}^{(z)}\rangle = \mathbb{E}\langle m_{t, p_1}^{(z)}, m_{t, p_2}^{(z)}\rangle = 0$ by the definition $m_{t-1} = (1 - \alpha_{t-1})m_{t-1}^{(y)} + \alpha_{t-1}m_{t-1}^{(z)}$. Hence we have $\mathbb{E}\langle m_{t, p_1}^{(y)}, m_{t, p_2}^{(y)}\rangle = \mathbb{E}\langle m_{t, p_1}^{(y)}, m_{t, p_2}^{(z)}\rangle = \mathbb{E}\langle m_{t, p_1}^{(z)}, m_{t, p_2}^{(z)}\rangle= 0$ for $t \in [T]$. Since $m_{t, p} = (1-\alpha_t)m_{t, p}^{(y)} + \alpha_t m_{t, p}^{(z)}$ for $p \in \{p_1, p_2\}$, we obtain $\mathbb{E}\langle m_{t, p_1}, m_{t, p_2}\rangle = 0$.
\end{proof}

\begin{pr_s_snag_ef_m_t_bound_prop}
Using Lemma \ref{lem: s_snag_ef_cum_error_expectation}, we have
\begin{align*}
    \mathbb{E}\|m_t\|^2 =&\ \mathbb{E}\|x_t - \widetilde{x}_t\|^2 \\
    =&\ \mathbb{E}\|(1-\alpha_t)(y_t - \widetilde{y}_t) + \alpha_t(z_t - \widetilde{z}_t)\|^2 \\
    \leq&\ (1-\alpha_t)^2\left(1 + \frac{\gamma}{2}\right)\mathbb{E}\|y_t - \widetilde{y}_t\|^2 + \frac{2\alpha_t^2}{\gamma} \mathbb{E}\|z_t - \widetilde{z}_t\|^2 \\
    =&\ (1 - \alpha_t)^2\left(1 + \frac{\gamma}{2}\right)\mathbb{E}\|m_{t}^{(y)}\|^2 + \frac{2\alpha_t^2}{\gamma}\mathbb{E}\|m_t^{(z)}\|^2.
\end{align*}
Using the definition of $\bar g_{t, p}^{(y)}$ and Lemma \ref{lem: s_snag_ef_cum_error_expectation}, similar to the proof of Proposition \ref{prop: s_sgd_ef_cum_error_recursive_bound}, we can show that 
\begin{align*}
    \mathbb{E}\|m_t^{(y)}\|^2
    =&\ \mathbb{E}\left\|x_{t-1} - \widetilde{x}_{t-1} + \eta_t\left(  \frac{1}{P}\sum_{p=1}^P\nabla f_{i_{t, p}}(x_{t-1}) - \frac{1}{P}\sum_{p=1}^P \bar g_{t, p}^{(y)} \right)\right\|^2 \\
    =&\ \mathbb{E}\left\|m_{t-1} + \eta_t\left( \frac{1}{P}\sum_{p=1}^P\nabla f_{i_{t, p}}(x_{t-1}) - \frac{1}{P}\sum_{p=1}^P \bar g_{t, p}^{(y)}\right)\right\|^2 \\
    \leq&\ (1-\gamma)\mathbb{E}\|m_{t-1}\|^2 + \Theta\left(\frac{\eta_t^2d}{kP}\left(\mathcal{V} + \mathbb{E}\|\nabla F(x_{t-1})\|^2\right)\right).
\end{align*}

On the other hand, by the definition of $z_t$ and $\widetilde{z}_t$, similarly we have
\begin{align*}
    \mathbb{E}\|m_t^{(z)}\|^2 =&\ \mathbb{E}\left\|(1 - \beta_t)m_{t-1}^{(z)} + \beta_t m_{t-1} + \lambda_t\left(  \frac{1}{P}\sum_{p=1}^P\nabla f_{i_{t, p}}(x_{t-1}) - \frac{1}{P}\sum_{p=1}^P \bar g_{t, p}^{(z)}\right)\right\|^2 \\
    \leq&\ (1 - \gamma)\mathbb{E}\|(1 - \beta_t)m_{t-1}^{(z)} + \beta_t m_{t-1}\|^2 + \Theta\left(\frac{\lambda_t^2d}{kP}\left(\mathcal{V} + \mathbb{E}\|\nabla F(x_{t-1})\|^2\right)\right) \\
    \leq&\ (1 - \gamma)(1 - \beta_t)\mathbb{E}\|m_{t-1}^{(z)}\|^2 + (1 - \gamma)\beta_t \mathbb{E}\|m_{t-1}\|^2 + \Theta\left(\frac{\lambda_t^2d}{kP}\left(\mathcal{V} + \mathbb{E}\|\nabla F(x_{t-1})\|^2\right)\right).
\end{align*}

Combining these inequalities, we get
\begin{align*}
    &\mathbb{E}\|m_t\|^2 + c_t\mathbb{E}\|m_t^{(z)}\|^2 \\
    \leq&\ ((1 - \gamma)(1 + \gamma/2)(1 -\alpha_t) + (1 - \gamma)(c_t + 2\alpha_t^2/\gamma)\beta_t)\mathbb{E}\|m_{t-1}\|^2 + (1 - \gamma)(c_t+2\alpha_t^2/\gamma)(1 - \beta_t)\mathbb{E}\|m_{t-1}^{(z)}\|^2 \\
    &+ \Theta\left(\frac{(\eta_t^2 + (c_t + \alpha_t^2/\gamma)\lambda_t^2)d}{kP}\left(\mathcal{V} + \mathbb{E}\|\nabla F(x_{t-1})\|^2\right)\right) \\
    \leq&\ ((1 - \gamma/2) + (c_t + 2\alpha_t^2/\gamma)\beta_t)\mathbb{E}\|m_{t-1}\|^2 + (1 - \gamma)(c_t + 2\alpha_t^2/\gamma)\mathbb{E}\|m_{t-1}^{(z)}\|^2 \\
    &+ \Theta\left(\frac{(\eta_t^2 + (c_t + \alpha_t^2/\gamma)\lambda_t^2)d}{kP}\left(\mathcal{V} + \mathbb{E}\|\nabla F(x_{t-1})\|^2\right)\right).
\end{align*}
for any positive sequence $\{c_t\}_{t=1}^T$.
Hence, if we set $c_t =  4\alpha_t^2/\gamma^2$ ($\leq 4\alpha_{t-1}^2/\gamma^2 = c_{t-1}$), $\beta_t \leq \gamma/(4(c_t + 2\alpha_t^2/\gamma) ) (\leq \Theta(\gamma^3/\alpha_t^2))$, we obtain
\begin{align*}
    &\mathbb{E}\|m_t\|^2 + c_t\mathbb{E}\|m_t^{(z)}\|^2 \\
    \leq&\ (1 - \gamma/4)(\mathbb{E}\|m_{t-1}\|^2 + c_{t-1}\mathbb{E}\|m_{t-1}^{(z)}\|^2) + \Theta\left(\frac{(\eta_t^2 + (\alpha_t^2/\gamma^2)\lambda_t^2)d}{kP}\left(\mathcal{V} + \mathbb{E}\|\nabla F(x_{t-1})\|^2\right)\right).
\end{align*}

Recursively using this inequality, we have
\begin{align*}
    \mathbb{E}\|m_t\|^2 + c_t\mathbb{E}\|m_t^{(z)}\|^2 \leq \Theta\left(\sum_{t'=1}^t \frac{(\eta_{t'}^2 + (\alpha_t^2/\gamma^2)\lambda_{t'}^2)d}{kP}(1 - \gamma/4)^{t - t'}(\mathcal{V} + \mathbb{E}\|\nabla F(x_{t'-1})\|^2\right).
\end{align*} \qed
\end{pr_s_snag_ef_m_t_bound_prop}

\subsection{Analysis for Convex Cases}
\begin{lem}\label{lem: s_snag_ef_object_bound}
Suppose that Assumptions \ref{assump: smoothness} and \ref{assump: strong_convexity} hold. For $x \in \mathbb{R}^d$, it follows that
\begin{align*}
    \mathbb{E}[F(\widetilde{y}_t)]\leq&\ F(x)  - \eta_t\left(\frac{3}{4} - \frac{\eta_t L}{2}\right)\|\nabla F(x_{t-1})\|^2 + \frac{\eta_t^2 L \mathcal{V}}{2P} +  \eta_tL^2\|x_{t-1} - \widetilde{x}_{t-1}\|^2 \\
    &- \langle \nabla F(x_{t-1}), x - \widetilde{x}_{t-1}\rangle - \frac{\mu}{2}\|\widetilde{x}_{t-1} - x\|^2 - \langle \nabla F(\widetilde{x}_{t-1}) - \nabla F(x_{t-1}), x - \widetilde{x}_{t-1}\rangle. 
\end{align*}
\end{lem}

\begin{proof}
By the $L$-smoothness of $F$ and the definition of $\widetilde{y}_t$, we have
\begin{align*}
    F(\widetilde{y}_t) \leq&\  F(\widetilde{x}_{t-1}) + \langle \nabla F(\widetilde{x}_{t-1}), \widetilde y_t - \widetilde x_{t-1} \rangle + \frac{L}{2}\|\widetilde{x}_{t-1} - \widetilde{y}_t\|^2 \\
    \leq&\ F(\widetilde x_{t-1}) -\eta_t\left\langle \nabla F(\widetilde{x}_{t-1}), \frac{1}{P}\sum_{p=1}^P \nabla f_{i_{t, p}}(x_{t-1})\right\rangle + \frac{\eta_t^2L}{2}\left\|\frac{1}{P}\sum_{p=1}^P \nabla f_{i_{t, p}}(x_{t-1})\right\|^2.
\end{align*}
Taking expectations of this inequality with respect to the $t$-th iteration gives
\begin{align*}
    &\mathbb{E}[F(\widetilde{y}_t)] \\
    \leq&\ F(\widetilde{x}_{t-1}) -\eta_t\langle \nabla F(\widetilde{x}_{t-1}) - \nabla F(x_{t-1}), \nabla F(x_{t-1})\rangle - \eta_t\left(1 - \frac{\eta_t L}{2}\right)\|\nabla F(x_{t-1})\|^2 + \frac{\eta_t^2L\mathcal{V}}{2P} \\
    \leq&\ F(\widetilde{x}_{t-1})  - \eta_t\left(\frac{3}{4} - \frac{\eta_t L}{2}\right)\|\nabla F(x_{t-1})\|^2 + \frac{\eta_t^2 L \mathcal{V}}{2P} + \eta_tL^2\|x_{t-1} - \widetilde{x}_{t-1}\|^2.
\end{align*}
Here the first inequality holds because $\mathbb{E}[(1/P)\sum_{p=1}^P\nabla f_{i_{t, p}}(x_{t-1})] = \nabla F(x_{t-1})$ and $\mathbb{E}\|(1/P)\sum_{p=1}^P\nabla f_{i_{t, p}}(x_{t-1})\|^2 = \mathbb{E}\|(1/P)\sum_{p=1}^P\nabla f_{i_{t, p}}(x_{t-1}) - \nabla F(x_{t-1})\|^2 + \|\nabla F(x_{t-1})\|^2 \leq \mathcal{V}/P + \|\nabla F(x_{t-1})\|^2$. The second inequality follows from $L$-smoothness of $F$ and Young's inequality. Also, we have
\begin{align*}
    F(\widetilde x_{t-1}) \leq&\ F(x) - \langle \nabla F(\widetilde{x}_{t-1}), x - \widetilde{x}_{t-1}\rangle - \frac{\mu}{2}\|\widetilde{x}_{t-1} - x\|^2 \\
    =&\ F(x) - \langle \nabla F(x_{t-1}), x - \widetilde{x}_{t-1}\rangle  - \frac{\mu}{2}\|\widetilde{x}_{t-1} - x\|^2 - \langle \nabla F(\widetilde{x}_{t-1}) - \nabla F(x_{t-1}), x - \widetilde{x}_{t-1}\rangle
\end{align*}
by $\mu$-strong convexity of $F$.
Combining the above two inequality results in 
\begin{align*}
    \mathbb{E}[F(\widetilde{y}_t)]\leq&\ F(x)  - \eta_t\left(\frac{3}{4} - \frac{\eta_t L}{2}\right)\|\nabla F(x_{t-1})\|^2 + \frac{\eta_t^2 L \mathcal{V}}{2P} +  \eta_tL^2\|x_{t-1} - \widetilde{x}_{t-1}\|^2 \\
    &- \langle \nabla F(x_{t-1}), x - \widetilde{x}_{t-1}\rangle - \frac{\mu}{2}\|\widetilde{x}_{t-1} - x\|^2 - \langle \nabla F(\widetilde{x}_{t-1}) - \nabla F(x_{t-1}), x - \widetilde{x}_{t-1}\rangle. 
\end{align*}
\end{proof}

\begin{lem}\label{lem: s_snag_ef_z_t_ineq} 
Suppose that Assumptions \ref{assump: smoothness} and \ref{assump: strong_convexity} hold. Set $\beta_t = \lambda_t\mu/(1 + \lambda_t\mu)$. For $x \in \mathbb{R}^d$, it follows that
\begin{align*}
    &- \langle \nabla F(x_{t-1}), x - \widetilde{x}_{t-1}\rangle - \frac{\mu}{2}\|\widetilde{x}_{t-1} - x\|^2 \\
    \leq&\ \frac{1}{\eta_t}\mathbb{E}\langle \widetilde{x}_{t-1} - \widetilde{y}_t, \widetilde{x}_{t-1} - \widetilde{z}_t \rangle \\
    &+ \frac{1-\beta_t}{2\lambda_t}\|\widetilde{z}_{t-1} -x\|^2 - (1 - \beta_t)\left(\frac{1}{2\lambda_t} + \frac{\mu}{2}\right)\mathbb{E}\|\widetilde{z}_t - x\|^2 - \frac{1-\beta_t}{2\lambda_t}\mathbb{E}\|\widetilde{z}_{t-1} - \widetilde{z}_t\|^2.
\end{align*}
\end{lem}
\begin{proof}
Let 
\begin{align*}
V_t(x) =&\ \frac{1+\lambda_t\mu}{\eta_t}\langle \widetilde{x}_{t-1} - \widetilde{y}_t, x - \widetilde{x}_{t-1} \rangle + \frac{1}{2\lambda_t}\|\widetilde{z}_{t-1} -x\|^2 + \frac{\mu}{2}\|\widetilde{x}_{t-1} - x\|^2 \\
=&\ \left(1 + \lambda_t\mu\right)\left\langle \frac{1}{P}\sum_{p=1}^P \nabla f_{i_{t, p}}(x_{t-1}), x - \widetilde{x}_{t-1}\right\rangle + \frac{1}{2\lambda_t}\|\widetilde{z}_{t-1} -x\|^2 + \frac{\mu}{2}\|\widetilde{x}_{t-1} - x\|^2.
\end{align*}
If we set $\beta_t = \lambda_t\mu/(1 + \lambda_t\mu)$, $\widetilde{z}_t$ is the minimizer of $V_t$ and $V_t$ is $1/\lambda_t + \mu$-strongly convex. Hence we have
\begin{align*}
    V_t(z_t) \leq&\ V_t(x) -\left(\frac{1}{2\lambda_t} + \frac{\mu}{2}\right)\|\widetilde{z}_t - x\|^2 \\
    =&\ \left(1 + \lambda_t\mu\right)\left\langle \frac{1}{P}\sum_{p=1}^P \nabla f_{i_{t, p}}(x_{t-1}), x - \widetilde{x}_{t-1}\right\rangle + \frac{1}{2\lambda_t}\|\widetilde{z}_{t-1} -x\|^2 - \left(\frac{1}{2\lambda_t} + \frac{\mu}{2}\right)\|\widetilde{z}_t - x\|^2 \\
    & + \frac{\mu}{2}\|\widetilde{x}_{t-1} - x\|^2.
\end{align*}
Using definition of $V_t$ and taking expectations of the both sides with respect to the $t$-th iteration yield
\begin{align*}
    &- \langle \nabla F(x_{t-1}), x - \widetilde{x}_{t-1}\rangle - \frac{\mu}{2}\|\widetilde{x}_{t-1} - x\|^2 \\
    \leq&\ \frac{1}{\eta_t}\mathbb{E}\langle \widetilde{x}_{t-1} - \widetilde{y}_t, \widetilde{x}_{t-1} - \widetilde{z}_t \rangle \\
    &+ \frac{1-\beta_t}{2\lambda_t}\|\widetilde{z}_{t-1} -x\|^2 - (1 - \beta_t)\left(\frac{1}{2\lambda_t} + \frac{\mu}{2}\right)\mathbb{E}\|\widetilde{z}_t - x\|^2 - \frac{1-\beta_t}{2\lambda_t}\mathbb{E}\|\widetilde{z}_{t-1} - \widetilde{z}_t\|^2.
\end{align*}
Here we used the relation $1/(1+\lambda_t\mu) = 1 - \beta_t$.
\end{proof}

\begin{prop}\label{prop: s_snag_ef_ideal_obj_gap_bound}
Suppose that Assumptions \ref{assump: sol_existence}, \ref{assump: smoothness}, \ref{assump: bounded_variance} and \ref{assump: strong_convexity} hold. Let $\eta_t = \eta \leq 1/(2L)$, $\lambda_t = \lambda = (1/2)\sqrt{\eta/\mu}$,  $\alpha_t = \alpha = \lambda\mu/(2 + \lambda\mu)$ and $\beta_t = \beta = \lambda\mu/(1 + \lambda\mu)$. Then S-SNAG-EF satisfies
\begin{align*}
    \mathbb{E}[F(\widetilde{y}_T) - F(x_*)] \leq \Theta\left(\mu(1-\alpha)^T + \sqrt{\frac{\eta}{\mu}}\frac{\mathcal{V}^2}{P} + \sum_{t=1}^T(1-\alpha)^{T-t}\left(\lambda L^2\mathbb{E}\|m_{t-1}\|^2 - \eta \mathbb{E}\|\nabla F(x_{t-1})\|^2\right)\right). 
\end{align*}
\end{prop}
\begin{proof}
Combining Lemma \ref{lem: s_snag_ef_object_bound} and Lemma \ref{lem: s_snag_ef_z_t_ineq} with $x = x_*$, we get
\begin{align}
    \begin{split}\label{ineq: s_snag_ef_obj_gap}
        \mathbb{E}[F(\widetilde{y}_t)] \leq&\ F(x_*)  - \eta_t\left(\frac{3}{4} - \frac{\eta_t L}{2}\right)\|\nabla F(x_{t-1})\|^2 + \frac{\eta_t^2L\mathcal{V}}{2P} + \eta_tL^2\|x_{t-1} - \widetilde{x}_{t-1}\|^2 \\
        &+ \frac{1}{\eta_t}\mathbb{E}\langle \widetilde{x}_{t-1} - \widetilde{y}_t, \widetilde{x}_{t-1} - \widetilde{z}_t \rangle \\
        &+ \frac{1-\beta_t}{2\lambda_t}\|\widetilde{z}_{t-1} -x_*\|^2 -  (1-\beta_t)\left(\frac{1}{2\lambda_t} + \frac{\mu}{2}\right)\mathbb{E}\|\widetilde{z}_t - x_*\|^2 -  \frac{1-\beta_t}{2\lambda_t}\mathbb{E}\|\widetilde{z}_{t-1} - \widetilde{z}_t\|^2 \\
        &- \langle \nabla F(\widetilde{x}_{t-1}) - \nabla F(x_{t-1}), x_* - \widetilde{x}_{t-1}\rangle \\
        =&\ F(x_*)  - \eta_t\left(\frac{3}{4} - \frac{\eta_t L}{2}\right)\|\nabla F(x_{t-1})\|^2 + \frac{\eta_t^2L\mathcal{V}}{2P} + \eta_tL^2\|x_{t-1} - \widetilde{x}_{t-1}\|^2 \\
        &+ \frac{1}{\eta_t}\mathbb{E}\langle \widetilde{x}_{t-1} - \widetilde{y}_t, \widetilde{x}_{t-1} - \widetilde{z}_t \rangle \\
        &+ \frac{1-\beta_t}{2\lambda_t}\|\widetilde{z}_{t-1} -x_*\|^2 -  (1-\beta_t)\left(\frac{1}{2\lambda_t} + \frac{\mu}{2}\right)\mathbb{E}\|\widetilde{z}_t - x_*\|^2 -  \frac{1-\beta_t}{2\lambda_t}\mathbb{E}\|\widetilde{z}_{t-1} - \widetilde{z}_t\|^2 \\
        &+ \langle \nabla F(\widetilde{x}_{t-1}) - \nabla F(x_{t-1}), \widetilde{x}_{t-1} - \widetilde{z}_t\rangle + \langle \nabla F(\widetilde{x}_{t-1}) - \nabla F(x_{t-1}), \widetilde{z}_t - x_*\rangle.
    \end{split}
\end{align}
Also, using Lemma \ref{lem: s_snag_ef_object_bound} with $x = \widetilde{y}_{t-1}$ gives
\begin{align}
    \begin{split}\label{ineq: s_snag_ef_obj_diff}
    \mathbb{E}[F(\widetilde{y}_t)] \leq&\ F(\widetilde{y}_{t-1})  - \eta_t\left(\frac{3}{4} - \frac{\eta_t L}{2}\right)\|\nabla F(x_{t-1})\|^2 + \frac{\eta_t^2L\mathcal{V}}{2P} + \eta_tL^2\|x_{t-1} - \widetilde{x}_{t-1}\|^2 \\
    &+ \frac{1}{\eta_t}\mathbb{E}\langle \widetilde{x}_{t-1} - \widetilde{y}_t, \widetilde{x}_{t-1} - \widetilde{y}_{t-1}\rangle + \langle \nabla F(\widetilde{x}_{t-1}) - \nabla F(x_{t-1}), \widetilde{x}_{t-1} - \widetilde{y}_{t-1}\rangle .
    \end{split}
\end{align}
Now, summing $\alpha_t \times$ (\ref{ineq: s_snag_ef_obj_gap}) and $(1 - \alpha_t) \times$ (\ref{ineq: s_snag_ef_obj_diff}) yields
\begin{align*}
    \mathbb{E}[F(\widetilde{y}_t)] - F(x_*) \leq&\ (1 - \alpha_t)(F(\widetilde{y}_{t-1}) - F(x_*)) \\
    &- \eta_t\left(\frac{3}{4} - \frac{\eta_t L}{2}\right)\|\nabla F(x_{t-1})\|^2 + \frac{\eta_t^2L\mathcal{V}}{2P} + \eta_tL^2\|x_{t-1} - \widetilde{x}_{t-1}\|^2 \\
    &+ \frac{1}{\eta_t}\mathbb{E}\langle \widetilde{x}_{t-1} - \widetilde{y}_t, \widetilde{x}_{t-1} - \alpha_t \widetilde{z}_t - (1 - \alpha_t)\widetilde{y}_{t-1}\rangle \\
    &+  \frac{\alpha_t(1-\beta_t)}{2\lambda_t}\|\widetilde{z}_{t-1} -x_*\|^2 -  \alpha_t(1-\beta_t)\left(\frac{1}{2\lambda_t} + \frac{\mu}{2}\right)\mathbb{E}\|\widetilde{z}_t - x_*\|^2 \\
    &-  \frac{\alpha_t(1-\beta_t)}{2\lambda_t}\mathbb{E}\|\widetilde{z}_{t-1} - \widetilde{z}_t\|^2 \\
    &+ \langle \nabla F(\widetilde{x}_{t-1}) - \nabla F(x_{t-1}), \widetilde{x}_{t-1} - \alpha_t \widetilde{z}_t - (1 - \alpha_t)\widetilde{y}_{t-1}\rangle \\
    &+ \alpha_t\langle \nabla F(\widetilde{x}_{t-1}) - \nabla F(x_{t-1}), \widetilde{z}_t - x_*\rangle.
\end{align*}
Since
$\widetilde{x}_{t-1} - \alpha_t \widetilde{z}_t - (1 - \alpha_t)\widetilde{y}_{t-1} = -\alpha_t(\widetilde{z}_t - \widetilde{z}_{t-1})$,
we have
\begin{align*}
    &\mathbb{E}\langle \widetilde{x}_{t-1} - \widetilde{y}_t, \widetilde{x}_{t-1} - \alpha_t \widetilde{z}_t - (1 - \alpha_t)\widetilde{y}_{t-1}\rangle \\
    =&\ - \alpha_t \mathbb{E}\langle \widetilde{x}_{t-1} - \widetilde{y}_t, \widetilde{z}_t - \widetilde{z}_{t-1}\rangle \\
    \leq&\ \frac{\alpha_t\lambda_t}{(1-\beta_t)\eta_t}\mathbb{E}\|\widetilde{x}_{t-1} - \widetilde{y}_t\|^2 + \frac{\alpha_t(1-\beta_t)\eta_t}{4\lambda_t}\|\widetilde{z}_t - \widetilde{z}_{t-1}\|^2 \\
    =&\ \frac{\alpha_t\lambda_t\eta_t}{1-\beta_t}\mathbb{E}\left\|\frac{1}{P}\sum_{p=1}^P \nabla f_{i_{t, p}}(x_{t-1})\right\|^2 + \frac{\alpha_t(1-\beta_t)\eta_t}{4\lambda_t}\|\widetilde{z}_t - \widetilde{z}_{t-1}\|^2 \\
    =&\ \frac{\alpha_t\lambda_t\eta_t}{1-\beta_t}\frac{\mathcal{V}}{P} + \frac{\alpha_t\lambda_t\eta_t}{1-\beta_t}\|\nabla F(x_{t-1})\|^2 + \frac{\alpha_t(1-\beta_t)\eta_t}{4\lambda_t}\|\widetilde{z}_t - \widetilde{z}_{t-1}\|^2. 
\end{align*}
Also, it holds that
\begin{align*}
    &\langle \nabla F(\widetilde{x}_{t-1}) - \nabla F(x_{t-1}), \widetilde{x}_{t-1} - \alpha_t \widetilde{z}_t - (1 - \alpha_t)\widetilde{y}_{t-1}\rangle \\
    =&\ - \alpha_t \langle \nabla F(\widetilde{x}_{t-1}) - \nabla F(x_{t-1}), \widetilde{z}_t -\widetilde{z}_{t-1}\rangle \\
    \leq&\ \frac{\lambda_t\alpha_t}{(1-\beta_t)}\|\nabla F(\widetilde{x}_{t-1}) - \nabla F(x_{t-1})\|^2 + \frac{\alpha_t(1 - \beta_t)}{4\lambda_t}\|\widetilde{z}_t - \widetilde{z}_{t-1}\|^2.
\end{align*}
Furthermore, we have
\begin{align*}
    \alpha_t\langle \nabla F(\widetilde{x}_{t-1}) - \nabla F(x_{t-1}), \widetilde{z}_t - x_*\rangle \leq \frac{\alpha_t}{(1 - \beta_t)\mu}\|\nabla F(\widetilde{x}_{t-1}) - \nabla F(x_{t-1})\|^2 + \alpha_t(1 - \beta_t)\frac{\mu}{4}\mathbb{E}\|\widetilde{z}_t - x_*\|^2.
\end{align*}
If we assume that 
\begin{equation}
    \frac{\alpha_t\lambda_t}{1-\beta_t} \leq \frac{\eta_t}{4}, \label{ineq: params_condition}
\end{equation}
by these inequality, we get
\begin{align*}
    \mathbb{E}[F(\widetilde{y}_t)] - F(x_*) \leq&\ (1 - \alpha_t)(F(\widetilde{y}_{t-1}) - F(x_*)) \\
    &+ \frac{(\eta_t^2L + \eta_t)\mathcal{V}}{2P} + \left(2\eta_tL^2 + \eta_tL^2/(\lambda_t\mu)\right)\|x_{t-1} - \widetilde{x}_{t-1}\|^2 - \frac{\eta_t}{4}\|\nabla F(x_{t-1})\|^2 \\
    &+  \frac{\alpha_t}{2\lambda_t}\|\widetilde{z}_{t-1} -x_*\|^2 -  \alpha_t\left(\frac{1}{2\lambda_t} + \frac{\mu}{4}\right)\mathbb{E}\|\widetilde{z}_t - x_*\|^2,
\end{align*}
Let $\Gamma_t = \Pi_{t'=1}^t (1 - \alpha_{t'}) > 0$. Multiplying $1/\Gamma_t$ to both sides of the above inequality yields
\begin{align*}
      \frac{1}{\Gamma_t}\mathbb{E}[F(\widetilde{y}_t)] - F(x_*) \leq&\ \frac{1}{\Gamma_{t-1}}(F(\widetilde{y}_{t-1}) - F(x_*)) \\
    &+ \frac{(\eta_t^2L + \eta_t)\mathcal{V}}{2\Gamma_t P} + \frac{2\eta_tL^2 + \eta_tL^2/(\lambda_t\mu)}{\Gamma_t}\|x_{t-1} - \widetilde{x}_{t-1}\|^2 - \frac{\eta_t}{8\Gamma_t}\|\nabla F(x_{t-1})\|^2 \\
    &+ \frac{\alpha_t(1 - \beta_t)}{2\Gamma_t \lambda_t}\|\widetilde{z}_{t-1} -x_*\|^2 -  \frac{\alpha_t(1 - \beta_t)}{\Gamma_t}\left(\frac{1}{2\lambda_t} + \frac{\mu}{4}\right)\mathbb{E}\|\widetilde{z}_t - x_*\|^2.  
\end{align*}
Taking expectations of this inequality with respect to the all random variables gives
\begin{align*}
      \frac{1}{\Gamma_t}\mathbb{E}[F(\widetilde{y}_t) - F(x_*)] \leq&\ \frac{1}{\Gamma_{t-1}}\mathbb{E}[F(\widetilde{y}_{t-1}) - F(x_*)] \\
    &+ \frac{(\eta_t^2L + \eta_t)\mathcal{V}^2}{2\Gamma_t P} + \frac{2\eta_tL^2 + \eta_tL^2/(\lambda_t\mu)}{\Gamma_t}\mathbb{E}\|x_{t-1} - \widetilde{x}_{t-1}\|^2 - \frac{\eta_t}{8\Gamma_t}\mathbb{E}\|\nabla F(x_{t-1})\|^2\\
    &+ \frac{\alpha_t(1 - \beta_t)}{2\Gamma_t \lambda_t}\mathbb{E}\|\widetilde{z}_{t-1} -x_*\|^2 -  \frac{\alpha_t(1 - \beta_t)}{\Gamma_t}\left(\frac{1}{2\lambda_t} + \frac{\mu}{4}\right)\mathbb{E}\|\widetilde{z}_t - x_*\|^2.  
\end{align*}
Let $\alpha_t = \alpha = \lambda\mu/(2 + \lambda\mu)$, $\beta = \beta_t = \lambda \mu/ (1 + \lambda\mu)$, $\eta_t = \eta$ and $\lambda_t = \lambda = (1/2)\sqrt{\eta/\mu}$. Then (\ref{ineq: params_condition}) holds. 
Also, we have the relation $\alpha_t/(2\Gamma_t \lambda_t) = \alpha/(2\Gamma_t\lambda) \leq \alpha/(2\Gamma_{t-1})(1/(2\lambda) + \mu/4)$ by the definition of $\alpha$.
Hence we have
\begin{align*}
      \frac{1}{\Gamma_t}\mathbb{E}[F(\widetilde{y}_t) - F(x_*)] \leq&\ \frac{1}{\Gamma_{t-1}}\mathbb{E}[F(\widetilde{y}_{t-1}) - F(x_*)] \\
    &+ \frac{(\eta^2L + \eta )\mathcal{V}^2}{2\Gamma_t P} + \frac{\eta L^2 + \eta L^2/(\lambda\mu)}{\Gamma_t}\mathbb{E}\|x_{t-1} - \widetilde{x}_{t-1}\|^2 - \frac{\eta}{8\Gamma_t}\mathbb{E}\|\nabla F(x_{t-1})\|^2\\
    &+ \frac{\alpha(1 - \beta)}{\Gamma_{t-1}}\left(\frac{1}{2\lambda} + \frac{\mu}{4}\right)\mathbb{E}\|\widetilde{z}_{t-1} -x_*\|^2 -  \frac{\alpha(1 - \beta)}{\Gamma_t}\left(\frac{1}{2\lambda} + \frac{\mu}{4}\right)\mathbb{E}\|\widetilde{z}_t - x_*\|^2.
\end{align*}
Summing up the above inequality from $t=1$ to $T$, we obtain
\begin{align*}
      \frac{1}{\Gamma_T}\mathbb{E}[F(\widetilde{y}_T) - F(x_*)] 
      \leq&\ \sum_{t=1}^T \frac{(\eta^2L + \eta)\mathcal{V}^2}{2\Gamma_t P} + \sum_{t=1}^T \frac{\eta L^2 + \eta L^2/(\lambda \mu)}{\Gamma_t}\mathbb{E}\|x_{t-1} - \widetilde{x}_{t-1}\|^2 \\
    &- \sum_{t=1}^T\frac{\eta}{8\Gamma_t}\mathbb{E}\|\nabla F(x_{t-1})\|^2+ \alpha(1-\beta)\left(\frac{1}{2\lambda} + \frac{\mu}{4}\right)\|\widetilde{z}_{0} -x_*\|^2.  
\end{align*}
Note that $\widetilde{z}_0 = x_0$. We also have 
$\alpha = \Theta(\sqrt{\eta\mu})$ since $\eta = O(1/L)$, $\mu \leq L$ and the setting of $\lambda$.
Multiplying $\Gamma_t$ to both sides of the inequality and rearranging it yield
\begin{align*}
    \mathbb{E}[F(\widetilde{y}_T) - F(x_*)] \leq \Theta\left(\mu(1-\alpha)^T + \sqrt{\frac{\eta}{\mu}}\frac{\mathcal{V}^2}{P} + \sum_{t=1}^T(1-\alpha)^{T-t}\left(\lambda L^2\mathbb{E}\|m_{t-1}\|^2 - \eta \mathbb{E}\|\nabla F(x_{t-1})\|^2\right)\right). 
\end{align*}
\end{proof}

\begin{lem}\label{lem: s_snag_ef_conversion}
\begin{align*}
    \mathbb{E}[F(x_{t-1}) - F(x_*)] \leq \Theta\left(\mathbb{E}[F(\widetilde{y}_t) - F(x_*)] + L\mathbb{E}\|m_{t-1}\|^2\right).
\end{align*}
\end{lem}

\begin{proof}
\begin{align*}
    F(x_{t-1}) \leq&\ F(\widetilde{x}_{t-1}) - \langle \nabla F(x_{t-1}), \widetilde{x}_{t-1} - x_{t-1}\rangle \\
    \leq&\ F(\widetilde{y}_t) - \langle \nabla F(\widetilde{x}_{t-1}), \widetilde{y}_t - \widetilde{x}_{t-1}\rangle - \langle \nabla F(x_{t-1}), \widetilde{x}_{t-1} - x_{t-1}\rangle \\
    =&\ F(\widetilde{y}_t) + \eta_t \left\langle \nabla F(\widetilde{x}_{t-1}), \frac{1}{P}\sum_{p=1}^P \nabla f_{i_{t, p}}(x_{t-1}) \right\rangle - \langle \nabla F(x_{t-1}), \widetilde{x}_{t-1} - x_{t-1}\rangle.
\end{align*}
Taking expectations of this inequality with respect to the $t$-th iteration gives
\begin{align*}
    F(x_{t-1}) \leq&\ \mathbb{E}[F(\widetilde{y}_{t})] + \eta_t \left\langle \nabla F(\widetilde{x}_{t-1}), \nabla F(x_{t-1}) \right\rangle - \langle \nabla F(x_{t-1}), \widetilde{x}_{t-1} - x_{t-1}\rangle \\
    \leq&\ \mathbb{E}[F(\widetilde{y}_{t})] - \eta_t \left\langle \nabla F(\widetilde{x}_{t-1}) - \nabla F(x_{t-1}), \nabla F(x_{t-1}) \right\rangle + \eta_t \|\nabla F(x_{t-1})\|^2 \\
    &+ \frac{1}{8L}\|\nabla F(x_{t-1})\|^2 + 2L\|x_{t-1} - \widetilde{x}_{t-1}\|^2 \\
    \leq&\ \mathbb{E}[F(\widetilde{y}_{t})] + 2\eta_t \|\nabla F(x_{t-1})\|^2 \\
    &+ \frac{1}{8L}\|\nabla F(x_{t-1})\|^2 + (\eta_tL^2+2L)\|x_{t-1} - \widetilde{x}_{t-1}\|^2 \\
    \leq&\  \mathbb{E}[F(\widetilde{y}_{t})] + \frac{3}{8L}\|\nabla F(x_{t-1})\|^2 + \Theta(L)\|m_{t-1}\|^2.
\end{align*}
Here the first inequality is due to the convexity of $F$. The second and third inequalities follow from Young's inequality and the $L$-smoothness of $F$. The last inequality holds because $\eta_t \leq 1/(8L)$. 
Also, by the $L$-smoothness and convexity of $F$, we have
$$\frac{1}{2L}\|\nabla F(x_{t-1})\|^2 \leq F(x_{t-1}) - F(x_*).$$ 
Using this inequality, we obtain
\begin{align*}
    \frac{1}{4}(F(x_{t-1}) - F(x_*)) \leq \mathbb{E}[F(\widetilde{y}_t) - F(x_*)] + \Theta(L)\|m_{t-1}\|^2.
\end{align*}
Multiplying $4$ to the both sides and taking expectations yields the claim.
\end{proof}

\begin{pr_s_sgd_ef_final_obj_gap_bound_prop}
Applying Lemma \ref{lem: s_snag_ef_conversion} with $t = T$ to Proposition \ref{prop: s_snag_ef_ideal_obj_gap_bound}, the statement can be immediately obtained. \qed
\end{pr_s_sgd_ef_final_obj_gap_bound_prop}

\begin{pr_s_snag_ef_strongly_convex_thm}
Let $\eta_t = \eta = \Theta(1/L \wedge \gamma^3P/L \wedge \gamma^2/\mu \wedge \gamma^{8/3}P^{2/3}\mu^{1/3}/L^{4/3} 
 \wedge P^2\mu \varepsilon^2/\mathcal{V}^2 \wedge \gamma^2\sqrt{\mu P\varepsilon}/(L\sqrt{\mathcal{V}})$.
At first, for using Proposition \ref{prop: s_snag_ef_m_t_bound}, it is required to be $\beta_t = \beta \leq \Theta(\gamma^3/\alpha^2)$. This condition is satisfied by assuming $\eta \leq \Theta(\gamma^2/\mu)$ be sufficiently small, because
\begin{align*}
    \beta = \frac{\lambda\mu}{1 + \lambda\mu} \leq \frac{\gamma^3}{\alpha^2} &\Leftrightarrow \frac{\sqrt{\eta\mu}}{2 + \sqrt{\eta\mu}} \leq \frac{\gamma^3}{\eta\mu} \\
     &\Leftarrow \sqrt{\eta\mu} \leq \frac{\gamma^3}{\eta\mu} \\
     &\Leftrightarrow \eta \leq \frac{\gamma^2}{\mu}.
\end{align*}
Observe that $1 - 2\alpha \geq 1 - \gamma$ can be satisfied for appropriate $\eta$, because $\alpha = \Theta(\sqrt{\eta\mu})$ and $\eta \leq \Theta(\gamma^2/\mu)$. Also, note that
\begin{align*}
    \frac{\lambda\eta^2L^2}{\gamma^4P} \leq \Theta(\eta) \Leftrightarrow \eta \leq \Theta(\gamma^{8/3}P^{2/3}\mu^{1/3}/L^{4/3})
\end{align*}
and
\begin{align*}
    \frac{\eta^2L}{\gamma^3P} \leq \Theta(\eta) \Leftrightarrow \eta \leq \Theta(\gamma^{3}P/L).
\end{align*}
From these facts, similar to the proof of Theorem \ref{thm: s_sgd_ef_strongly_convex}, combining Proposition \ref{prop: s_snag_ef_final_obj_gap_bound} and Proposition \ref{prop: s_snag_ef_m_t_bound}, we have
\begin{align*}
      &\mathbb{E}[F(x_{\mathrm{out}}) - F(x_*)] \\
      \leq&\ \Theta\left(\mu(1-\alpha)^T\|x_0 -x_*\|^2 + \sqrt{\frac{\eta}{\mu}}\frac{\mathcal{V}}{P} 
      + \sum_{t=1}^T (1-\alpha)^{T-t}\left(\lambda L^2\mathbb{E}\|m_{t-1}\|^2 - \eta \mathbb{E}\|\nabla F(x_{t-1})\|^2\right) + L\mathbb{E}\|m_{T-1}\|^2 \right) \\
      =&\ \Theta\left(\mu(1-\alpha)^T\|x_0 -x_*\|^2 + \sqrt{\frac{\eta}{\mu}}\frac{\mathcal{V}}{P} + \frac{\lambda\eta^2L^2\mathcal{V}d}{\gamma^2 kP} \sum_{t=1}^T(1-\alpha)^{T-t}\sum_{t'=1}^t (1 - \gamma)^{t-t'} \right.\\
      &\left.+ \sum_{t=1}^T \frac{\lambda\eta^2L^2d}{\gamma^2kP}(1-\alpha)^{T-t}\sum_{t'=1}^t(1 - \gamma)^{t-t'}\mathbb{E}\|\nabla F(x_{t'-1})\|^2 - \sum_{t=1}^T(\eta (1 - \alpha)^{T-t}/2)\mathbb{E}\|\nabla F(x_{t-1})\|^2\right. \\
      &\left. + \frac{\eta^2L\mathcal{V}d}{\gamma^2kP}\sum_{t=1}^{T}(1 - \gamma)^{T-t} + \sum_{t=1}^T \frac{\eta^2Ld}{\gamma^2kP}(1 - \gamma)^{T-t} \mathrm{E}\|\nabla F(x_{t-1})\|^2- \sum_{t=1}^T (\eta(1 - \alpha)^{T-t}/2)\mathrm{E}\|\nabla F(x_{t-1})\|^2\right) \\
      \leq&  \Theta\left(\mu(1-\alpha)^T\|x_0 -x_*\|^2 + \sqrt{\frac{\eta}{\mu}}\frac{\mathcal{V}}{P} + \frac{\lambda \eta^2L^2\mathcal{V}}{\gamma^4\alpha P} + \frac{\eta^2L\mathcal{V}}{\gamma^4P}\right) \\
      \leq&\ \Theta\left(\mu(1-\alpha)^T\|x_0 -x_*\|^2 + \sqrt{\frac{\eta}{\mu}}\frac{\mathcal{V}}{P} + \frac{\lambda \eta^2L^2\mathcal{V}}{\gamma^4\alpha P}\right).
\end{align*}
Here, the last inequality is due to $\lambda L/\alpha \geq 1$.
Set appropriate $T = \widetilde\Theta(1/\alpha) = \widetilde\Theta(1/\sqrt{\eta\mu})$. The sufficient conditions for $\mathbb{E}[F(x_{\mathrm{out}}) - F(x_*)] \leq \varepsilon$ are 
 $\eta \leq \Theta(1/L)$, $\eta \leq \gamma^2/\mu$, $\eta \leq \Theta(\gamma^{8/3}P^{2/3}\mu^{1/3}/L^{4/3})$, $\eta \leq \Theta(\gamma^{3}P/L)$, $\sqrt{\eta/\mu}\mathcal{V}/P \leq \varepsilon$ and $\lambda \eta^2L^2\mathcal{V}/(\gamma^4\alpha P) \leq \Theta(\varepsilon)$. Substituting the definition of $\eta$ to $\widetilde{O}(1/\sqrt{\eta\mu})$, we obtain the desired result. \qed
\end{pr_s_snag_ef_strongly_convex_thm}

\subsection{Analysis for Nonconvex Cases}

\begin{pr_reg_s_snag_ef_general_nonconvex_thm}
Let $\eta_t = \eta$ be the same one defined in the proof of Theorem \ref{thm: s_snag_ef_strongly_convex}.
First observe that $F_s = F + \sigma Q(x_{s-1})$ is $3L$-smooth and $L$-strongly convex, since $\sigma = L$ and $F$ is $L$-smooth. Also note that $\{f_{i, p} + \sigma Q(x_{s-1}^{+})\}_{i, p}$ has $\mathcal{V}$-bounded variance. 
From Theorem \ref{thm: s_snag_ef_strongly_convex} (with $\mu \leftarrow L$ and $\varepsilon \leftarrow \varepsilon/(16L)$), we have
$$\mathbb{E}\|\nabla F_s(x_s)\|^2 = \mathbb{E}\left\|\nabla F_s\left(x_T^{(s-1)}\right)\right\|^2 \leq 2L\mathbb{E}\left[F_s\left(x_T^{(s-1)}\right) - \mathrm{min}_{x \in \mathbb{R}^d} F_s(x)\right] \leq \frac{\varepsilon}{8}, $$
with iteration complexity
\begin{align*}
    \widetilde O\left(1 + \frac{\mathcal{V}}{P}\frac{1}{\varepsilon} + \frac{d}{k} + \frac{d^{\frac{3}{2}}}{k^{\frac{3}{2}}\sqrt{P}} + \frac{d^{\frac{4}{3}}}{k^{\frac{4}{3}}P^{\frac{1}{3}}} + \frac{d}{kP^{\frac{1}{4}}}\frac{\mathcal{V}^{\frac{1}{4}}}{\varepsilon^{\frac{1}{4}}}\right),
\end{align*}
Using this fact, we have
\begin{align*}
    \mathbb{E}\|\nabla F(x_s)\|^2 \leq&\ 2\mathbb{E}\|\nabla F_t(x_s)\|^2 + 4L^2\mathbb{E}\|x_s - x_{s-1}\|^2 \\
    \leq&\ \frac{\varepsilon}{4} + 4L^2\mathbb{E}\|x_s - x_{s-1}\|^2.
\end{align*}
Now we need to bound $\mathbb{E}\|x_s - x_{s-1}\|^2$.
\begin{align*}
    \mathbb{E}[F(x_s)] =&\ \mathbb{E}[F_s(x_s)] - L\mathbb{E}\|x_s - x_{s-1}\|^2 \\
    =&\ \mathbb{E}[F_s(x_s) - \mathrm{min}_{x \in \mathbb{R}^d}F_s(x)] + \mathbb{E}[\mathrm{min}_{x \in \mathbb{R}^d}F_s(x)] - L\mathbb{E}\|x_s - x_{s-1}\|^2 \\
    \leq&\ \frac{\varepsilon}{16L} + \mathbb{E}[F_t(x_{s-1})] - L\mathbb{E}\|x_s - x_{s-1}\|^2 \\
    =&\ \frac{\varepsilon}{16L} + \mathbb{E}[F(x_{s-1})] - L\mathbb{E}\|x_s - x_{s-1}\|^2.
\end{align*}
Hence we obtain
\begin{align}\label{ineq: s_sgd_ef_reg_nonconvex_grad_bound}
    \mathbb{E}\mathbb\|\nabla F(x_s)\|^2 \leq \frac{\varepsilon}{2} + 4L\mathbb{E}[F(x_{s-1}) - F(x_s)].
\end{align}
Summing this inequality from $s=1$ to $S$ and divide the result by $S$ yield
$$\mathbb{E}\|\nabla F(x_{\mathrm{out}})\|^2 \leq \frac{\varepsilon}{2} + \Theta(L)\frac{F(x_{\mathrm{in}}) - F(x_*)}{S}.$$
Therefore appropriately large $S = \Theta(1 +  L\Delta/\varepsilon)$ is sufficient for ensuring $\mathbb{E}\|\nabla F(x_{\mathrm{out}})\|^2 \leq \varepsilon$. \qed
\end{pr_reg_s_snag_ef_general_nonconvex_thm}

\end{document}